\documentclass[12pt]{amsart}
\usepackage{amsfonts}
\usepackage{amsmath}
\usepackage{amssymb}
\usepackage[margin=1.2in]{geometry}
\usepackage{verbatim}
\usepackage{xcolor}
\usepackage{accents}
\newcommand{\ubar}[1]{\underaccent{\bar}{#1}}

\usepackage{algorithm}
\usepackage{algpseudocode}

\newtheorem{theorem}{Theorem}[section]

\newtheorem{lemma}{Lemma}[section]

\theoremstyle{remark}
\newtheorem{remark}[theorem]{Remark}

\def \vector#1{\boldsymbol{#1}}
\newcommand {\D} {\displaystyle}
\usepackage{bm}

\newcommand{\half}{\frac{1}{2}}
\newcommand {\imbed} {\hookrightarrow}

\newcommand {\OO}[1]{{\mathcal O}\left(#1\right)}

\newcommand {\scal}[2]{\left(#1,#2\right)}

\newcommand{\abs}[1]{\left\lvert #1 \right\rvert}
\newcommand{\vnorma}[1]{\left\|#1\right\|}
\newcommand {\meas}[1] {\left|#1\right|}
\DeclareMathOperator{\di}{d\hspace{-1.5pt}}

\newcommand{\dX}{ \di \X}
\newcommand {\dt}{ \di t}
\newcommand{\ds}{\di s}

\newcommand {\sigmab}{\vector{\sigma}}
\newcommand {\dg}{ \di \sigmab}

\newcommand {\NN } {{\mathbb N}}
\newcommand {\RR } {{\mathbb R}}
\newcommand{\veps}{\varepsilon}

\newcommand {\pdt}{\partial_t}

\newcommand {\I}{[0,T]}
\newcommand {\Iopen}{(0,T)}

\newcommand {\domain}{\Omega}
\newcommand{\vct}[1]{\boldsymbol{#1}}
\newcommand {\normal}{\vct{\nu}}
\newcommand {\lp}[1]{\Leb^{#1} (\domain )}

\DeclareMathOperator{\Leb}{L}

\DeclareMathOperator{\Cont}{C}
\DeclareMathOperator{\Hi}{H}

\newcommand {\hk}[1]{\Hi^{#1}(\domain )}

\newcommand {\lpkIX}[2]{\Leb^{#1}\left(\Iopen,#2\right)}

\newcommand {\cIX}[1]{\Cont\left(\I,#1\right)}

\newcommand {\X}{\vector{x}}

\usepackage{comment}
\usepackage[shortlabels]{enumitem}

\makeatletter
\@namedef{subjclassname@2020}{%
  \textup{2020} Mathematics Subject Classification}
\makeatother

\usepackage{xcolor}
\usepackage[hidelinks]{hyperref}
\hypersetup{
    colorlinks = true,
    linkcolor = red,
    anchorcolor = red,
    citecolor = green,
    filecolor = red,
    urlcolor = red,
    linktoc=page
    }
\usepackage{cleveref}

\usepackage{graphicx}
\usepackage{subfigure}
\usepackage{float}
\usepackage{caption}

\begin{document}

	\title[ISP semilinear pseudo-parabolic equation]{A time-dependent inverse source problem for a semilinear pseudo-parabolic equation with Neumann boundary condition}

	\author[K.~Van~Bockstal]{Karel Van Bockstal$^1$} 
	\thanks{This work was supported by grant no.~AP23486218
 of the Ministry of  Science and High Education of the Republic
of Kazakhstan (MES RK) and the Methusalem programme of Ghent University Special Research Fund (BOF) (Grant Number 01M01021)}

	\address[1]{Ghent Analysis \& PDE center, Department of Mathematics: Analysis, Logic and Discrete Mathematics\\ Ghent University\\
		Krijgslaan 281\\ B 9000 Ghent\\ Belgium} 
	\email{karel.vanbockstal@UGent.be}

 \author[K.~Khompysh]{Khonatbek Khompysh$^{1,2}$} 

 \address[2]{Al-Farabi Kazakh National University\\ Almaty\\ Kazakhstan}
 \email{konat\_k@mail.ru}
	
	
	\begin{abstract}  
      In this paper, we study the inverse problem for determining an unknown time-dependent source coefficient in a semilinear pseudo-parabolic equation with variable coefficients and Neumann boundary condition. This unknown source term is recovered from the integral measurement over the domain $\Omega$. Based on Rothe's method, the existence and uniqueness of a weak solution, under suitable assumptions on the data, is established. A numerical time-discrete scheme for the unique weak solution and the unknown source coefficient is designed, and the convergence of the approximations is proven. Numerical experiments are presented to support the theoretical results. Noisy data is handled through polynomial regularisation.
	\end{abstract}
	
	\maketitle
	
	\tableofcontents

%%%%%%%%%%%%%%%%%%%%%
\section{Introduction}
\label{sec:introduction}
%%%%%%%%%%%%%%%%%%%%

We consider an open and bounded Lipschitz domain $\Omega \subset \mathbb{R}^d$ with boundary $\partial\Omega$,  $d \in \NN$. 
 Let $Q_T = (0,T] \times \Omega$ and $\Sigma_T = (0,T] \times \partial\Omega$, where $T>0$ is a given final time.
In the sequel, we consider the following  semilinear pseudo-parabolic equation with variable coefficients and Neumann boundary condition: 
\begin{equation} \label{eq:problem}
\left\{
\begin{array}{rl}
\pdt u(t,\X) -   \nabla \cdot\left(\eta(t,\X) \nabla \pdt u(t,\X)\right)   - \nabla \cdot \left( \kappa(t,\X) \nabla u (t,\X)\right) & \\ 
= F(u(t,\X))+\D p(t,\X)h(t),  & (t,\X) \in Q_T, \\[4pt]
\kappa(t,\X) \nabla u(t,\X) \cdot \normal(\X) = g(t,\X), & (t,\X) \in \Sigma_T, \\[4pt]
u(0,\X) = \tilde{u}_0,  &\X \in \Omega.
\end{array}
\right.
\end{equation}
In this paper, $h(t)$ is unknown and will be recovered from the additional information 
\begin{equation} \label{eq:add:cond}
\int_\Omega  u(t,\X) \dX = m(t), \quad t > 0. 
\end{equation}
In the system (\ref{eq:problem}-\ref{eq:add:cond}), the coefficients $ \eta, \kappa $ and the functions $\tilde{u}_0, p, g, m, F$ are given, whilst $u$ and $h$ are unknown and need to be determined. We look for a weak solution to the problem (\ref{eq:problem}-\ref{eq:add:cond}). 

Equations like \eqref{eq:problem} are called pseudo-parabolic equations, also known as Sobolev equations. These equations describe a range of essential physical processes, such as the unidirectional propagation of nonlinear long waves \cite{Ting:1963,BBM:1972}, the aggregation of populations \cite{Pad:2004}, fluid flow in fissured rock \cite{BZK:1960}, filtration in porous media \cite{6BER:1989}, the unsteady flow of second-order fluids \cite{Hui:1968}, and the motion of non-Newtonian fluids \cite{AKS2011,zvy-2010}, among others.

The measurement \eqref{eq:add:cond} in integral form naturally arises \cite{Cannon1986, POV:2000} and it describes the total or average value of $u$ over the entire domain.  It serves as supplementary information for determining the source term, and it has significant physical meaning.  This condition is particularly relevant in practical applications where local pointwise or instantaneous temperature measurements are prone to large errors or even infeasible (e.g., at high temperatures). In such cases, the global-in-space measurement \eqref{eq:add:cond} provides a more realistic and reliable source of information. For instance, such a condition has been used in studying various physical phenomena, including chemical engineering \cite{cannon1987galerkin,cannon1986diffusion}, thermoelasticity \cite{shi1992design}, heat conduction and diffusion processes \cite{ionkin1977solution,kamynin1964boundary,cannon1963,GRIMMONPREZ2015331} and fluid flow in porous media \cite{ewing1991class}. For example, in \cite{ionkin1977solution}, this condition in problems related heat conduction modeling, indicates the total heat $m(t)$. A similar formulation occurs in thermoelasticity, where the design of periodic contact with a prescribed pattern leads to a heat conduction problem with a given energy input \cite{shi1992design}.

Numerous studies on linear and nonlinear pseudo-parabolic equations have focused on investigating direct problems. Research on inverse problems for pseudo-parabolic equations began with the seminal work of Rundell \cite{R:1980} in 1980, wherein Rundell studied an inverse problem to identify source terms in a linear pseudo-parabolic equation using overspecified boundary data or the final-in-time measurement. Currently, there are dozens of studies on inverse problems for pseudo-parabolic equations; we refer readers to \cite{Fu2023,Fu2023a,AntAA:2020,KhSh:2023,Yaman2012,LyTa:2011,LyVe:2019,HDT:2021,HTI:2024,KHSI:2024} and the references therein.  In \cite{Fu2023}, the authors study the determination of a time-dependent potential coefficient from the nonlocal measurement  $\int_\Omega u(t,\X) \omega(\X)\dX$ ($\left.\omega\right|_{\Gamma} = 0$) in a linear pseudo-parabolic with constant coefficients. Using a fixed point argument, the authors establish the uniqueness of a solution if $\vnorma{\nabla \omega}_{\lp{2}}\ll 1.$ In \cite{Fu2023a}, authors have been obtained similar results also for an inverse problem recovering a spatial-dependent potential coefficient in a linear pseudo-parabolic equation with nonlocal measurement $\int_0^T u (t,\X)\omega(t)\dt$ under some restriction on the positive weight function $\omega$ such as $0<\frac{\|\omega'\|_{\Leb^2(0,T)}}{\|\omega\|_{\Leb^1(0,T)}}\ll 1.$ In \cite{LyTa:2011}, the authors studied an inverse problem of determining the diffusion coefficient $\kappa = \kappa(t)$ in a linear pseudo-parabolic equation
under a measurement on the boundary.
Under certain assumptions and restrictions on the data, the existence and uniqueness of a strong solution were established using an iterative method. Moreover, the inverse problem of determining a time-dependent potential coefficient from integral boundary data is investigated in \cite{LyVe:2019}. The numerical analysis of inverse problems involving the determination of a time-dependent potential coefficient in a one-dimensional linear pseudo-parabolic equation, based on additional boundary information about the solution, is presented in \cite{HDT:2021}. A similar investigation is carried out for the case of a fractional time derivative in \cite{HTI:2024}. We note that numerical and theoretical studies of a spatially dependent inverse source problem for a pseudo-parabolic equation with memory are addressed in \cite{KHSI:2024}.

In the study of time-dependent inverse (source) problems for pseudo-parabolic equations, the measurement form plays a crucial role. For instance, in works such as \cite{AntAA:2020} and \cite{KhSh:2023}, inverse problems for nonlinear pseudo-parabolic equations perturbed by $p-$ Laplacian and nonlinear damping/reaction term have been investigated under a measurement expressed in a specific form, such as $\int_\Omega u(\omega - \Delta \omega)\dX = e(t)$. With this specific nonlocal measurement, the blowing up in a finite time and large time behaviour of solutions of an inverse source problem for a pseudo-parabolic equation with a nonlinear damping term were established in \cite{Yaman2012}. Recently, in \cite{Khompysh2025}, an inverse problem for the linear pseudo-parabolic equation with a memory term was considered under the overdetermination condition in a similar specific form and the authors established the existence of strong solutions, in particular, $h\in \Cont([0,T])$, and  in the 1D case, the numerical solutions were explored by using $B$-spline collocation technique.
However, as we have mentioned above, this measurement form lacks a clear physical justification, although it is mathematically convenient. To the best of our knowledge, there are very few studies on inverse source problems for pseudo-parabolic equations that recover a time-dependent source under the condition $\int_\Omega u(t,\X) \omega(\X)\dX$, in particular \eqref{eq:add:cond}, even in the linear case, we may refer just to \cite{Fu2023}. Therefore, in this work, we study a time-dependent inverse source problem for a semilinear pseudo-parabolic equation under the measurement \eqref{eq:add:cond}.

Various applications of Rothe's method \cite{Rothe1930} for the study of inverse problems for parabolic and hyperbolic evolution equations have been considered by Slodi\v{c}ka and Van Bockstal, see e.g. \cite{Slodicka2014jcam,Slodicka2016b,VanBockstal2017,Kang2018,Siskova2019,VanBockstal2020,VanBockstal2022a,VanBockstal2022b,VanBockstal2022c}. 
In this work, we apply Rothe's method to investigate the inverse source problem for the semilinear pseudo-parabolic equation \eqref{eq:problem} with Neumann boundary condition and the measurement \eqref{eq:add:cond} from both theoretical and numerical perspectives. Specifically, we will establish the existence and uniqueness of a weak solution and develop a numerically efficient algorithm.

We stress that the analytical techniques employed in this paper (Rothe’s method, i.e. time-discretisation, a priori estimates and compactness arguments) are classical and closely related to those in previous works on time-dependent inverse problems solved with the aid of Rothe's method \cite{Slodicka2014jcam,Slodicka2016b,VanBockstal2017,Kang2018,Siskova2019,VanBockstal2020,VanBockstal2022a,VanBockstal2022b,VanBockstal2022c}. However, the main novelty of our study lies in the form of the measurement condition: instead of using the problem-specific integral measurement 
$\int_\Omega u (\omega-\Delta\omega) \dX = e(t)$, which lacks a clear physical motivation. Our contribution is that the inverse source problem with the classical measurement \eqref{eq:add:cond} for semilinear pseudo-parabolic equations with Neumann boundary data remains well-posed for exact data (without any restrictive condition on the data), and can be treated within the established Rothe framework, and that this setting also allows for efficient numerical implementation.
 
The paper is organised as follows. In \Cref{sec:reformulation}, we first state all conditions on the given data, reformulate the inverse problem as a coupled direct problem and formulate its weak formulation. Afterwards, we show the uniqueness of a solution in \Cref{sec:uniqueness} and the existence of the weak solution in \Cref{sec:existence} using Rothe's method. In \Cref{sec:experiments}, the theoretical results are illustrated with some numerical examples in 1D and 2D cases.

%%%%%%%%%%%%%%%%%%%%%
\section{Reformulation of the inverse problem}
\label{sec:reformulation}
%%%%%%%%%%%%%%%%%%%%

In this section, we will derive an expression for $h$ in terms of $u$ and the given data. In this way, we can reformulate the inverse problem as a coupled direct problem. We first summarise the assumptions on the data that we will use to tackle the inverse problem  (\ref{eq:problem}-\ref{eq:add:cond}):
\begin{enumerate}[
\textbf{AS DP}-(1),leftmargin=2.4cm] %\roman*
           \item\label{as:DP:eta} $\eta: \I\times \overline{\Omega} \to \RR$ satisfies $  0 < \ubar{\eta}_0 \le \eta(t,\X)\le \ubar{\eta}_1<\infty;$ 
    \item \label{as:DP:kappa} $\kappa: \I\times \overline{\Omega} \to \RR$ satisfies
    \[
    \begin{cases}
    0<\ubar{\kappa}_0\le \kappa(t,\X)\le \ubar{\kappa}_1<\infty, \ &\text{for a.a. } (t,\X)\in [0,T]\times\overline{\Omega},\\
    \abs{\pdt \kappa(t,\X)}\le \ubar{\kappa}^\prime_1<\infty,\ &\text{for a.a. } (t,\X)\in [0,T]\times \overline{\Omega};
    \end{cases}
    \]
     \item \label{as:DP:f} $F:\mathbb{R} \rightarrow \mathbb{R}$ is  Lipschitz continuous, i.e. 
    \[
     \abs{F(s_1)-F(s_2)}\le L_F \abs{s_1-s_2}, \quad \forall s_1,s_2\in \RR; 
    \]
    \item \label{as:DP:u0} $\tilde{u}_0\in\hk{1};$
    \item \label{as:g} $g \in \Hi^1 \left((0,T], \Leb^2(\partial\Omega)\right)$, so that
    \[
    G := \frac{g}{\kappa} \in \Hi^1 \left((0,T], \Leb^2(\partial\Omega)\right); \ \ 
    \]
    \item \label{as:p} $p\in \mathcal{X}:=\cIX{\lp{2}}$ such that ${\omega}\in \Cont\left(\I\right)$ defined  by 
\[{\omega}(t):=\int_\Omega p(t,\X)\dX\]
satisfies
\[
 {\omega}(t) \neq 0 \text{ for all } t\in \I.
\]
We denote 
\[
 0< \ubar{\omega}_0:=\min_{t\in \I} \abs{{\omega}(t)};% \ \ 
\]
  \item \label{as:m}  $m\in \Hi^1((0,T])$. 
\end{enumerate}

\begin{remark}
Further, we denote the inner product $\scal{\cdot}{\cdot}_X$ by $\scal{\cdot}{\cdot}$ for $X= \lp{2}$ and by $\scal{\cdot}{\cdot}_{\partial\Omega}$  for $X= \Leb^2(\partial\Omega)$. Its associated norm is denoted by $\vnorma{\cdot} = \sqrt{\scal{\cdot}{\cdot}}$ and $\vnorma{\cdot}_{\partial \Omega} = \sqrt{\scal{\cdot}{\cdot}}_{\partial\Omega}$, respectively. 
\end{remark}

\begin{remark}
    From \ref{as:DP:f} it follows that 
\begin{equation}\label{eq:inequality_nonlinear_f}
    \abs{F(s)} \le \abs{F(0)} + L_F \abs{s}, \quad \forall s\in \RR.
    \end{equation}
\end{remark}

\begin{remark}
    In \Cref{thm:existence_inverse_problem}, we will show that $u: \Iopen \to \hk{1}$ is continuous in time. Together with \ref{as:m}, this implies that $m(0)=\int_\Omega\tilde{u}_0(\X)\dX.$
\end{remark}

The approach presented here takes advantage of the Neumann boundary data. For this reason, when integrating the PDE in \eqref{eq:problem} over $\Omega$, using the measurement \eqref{eq:add:cond} and the divergence theorem, we obtain the following expression for $h$ (with $t>0$): 
\begin{multline} \label{eq:expression_h}
h(t)=\frac{1}{{\omega}(t)}\left[m^\prime(t) - \int_{\partial\Omega} \eta(t,\X)\partial_t G(t,\X) \dg_{\X} \right.\\\left. - \int_{\partial\Omega} g(t,\X) \dg_{\X}  - \int_\Omega F(u(t,\X))\dX\right],    
\end{multline}
where we have used \ref{as:p}. Using this expression, we can reformulate the inverse problem  (\ref{eq:problem}-\ref{eq:add:cond}) in the following way: 
\medskip
\begin{center}
Find $u(t)\in \hk{1}$ with $\pdt u(t)\in \hk{1}$ such that for a.a. $t \in \Iopen$ and any $\varphi \in \hk{1}$ it holds that 
\begin{multline}\label{eq:var_for} 
\scal{ \pdt u(t)}{\varphi} + \scal{ \eta(t)\nabla \pdt u(t)}{\nabla \varphi} +  \scal{\kappa(t) \nabla u(t)}{\nabla \varphi} \\
=h(t) \scal{p(t)}{\varphi} + \scal{F(u(t))}{\varphi} +  \scal{ \eta(t)\partial_t G(t)}{\varphi}_{\partial\Omega} + \scal{ g(t)}{\varphi}_{\partial\Omega},    
\end{multline}
with $h\in \Leb^2(0,T)$ given by \eqref{eq:expression_h}.
\end{center}
\medskip

In the next section, we show the uniqueness of a solution to the problem (\ref{eq:expression_h}-\ref{eq:var_for}).

%%%%
\section{Uniqueness of a solution}
\label{sec:uniqueness}
%%%%

 We show the uniqueness of a solution to the problem (\ref{eq:expression_h}-\ref{eq:var_for}) by the energy estimate approach.

%%%%%%%%%%%%%%%%%%%%%
%\subsection{Uniqueness of a solution}

\begin{theorem}\label{thm:uniq_inv_problem}
Let the assumptions \ref{as:DP:eta} until \ref{as:m} be fulfilled. 
    Then, there exists at most one couple $\{u,h\}$ solving problem (\ref{eq:expression_h}-\ref{eq:var_for})  such that
\begin{equation*}
h \in \Leb^2\Iopen, \quad  u \in \cIX{\hk{1}} \quad \text{with} \quad  \pdt u \in \lpkIX{2}{\hk{1}}.
\end{equation*}
\end{theorem}

\begin{proof}
Let $u_1$ and $u_2$ be two distinct solutions to the problem (\ref{eq:expression_h}-\ref{eq:var_for}) with the same data. 
We subtract the variational formulation \eqref{eq:var_for}  for $\{u_1, h_1\}$ from the one for $\{u_2,h_2\}$. Then, we obtain  for $u:=u_1-u_2$ and $h:=h_1-h_2$  that
\begin{multline}\label{uni:var_for}
\scal{\pdt u(t)}{\varphi} + \scal{\eta(t) \nabla \pdt u(t)}{\nabla \varphi} +  \scal{\kappa(t) \nabla u(t)}{\nabla \varphi} \\
=h(t) \scal{p(t)}{\varphi} + \scal{F(u_1(t))-F(u_2(t))}{\varphi}, \quad \forall \varphi \in \hk{1}.  
\end{multline}
Performing the similar operator for \eqref{eq:expression_h}, we obtain that $h$ is expressed as follows
\begin{equation} \label{uni:expression_h}
h(t)= 
\frac{1}{{\omega}(t)}\int_\Omega \left[F(u_2(t,\X))-F(u_1(t,\X))  \right]\dX.\\
\end{equation}
Employing the Lipschitz continuity of $F$, we obtain that 
\begin{equation}\label{proof:uniq:est_h1}
\abs{h(t)} \le \frac{L_F}{\ubar{\omega}_0}  \vnorma{u(t)}_{\Leb^1(\Omega)}\le C_1\vnorma{u(t)}, \ \ C_1:=\frac{L_F}{\ubar{\omega}_0}\sqrt{\meas{\Omega}}.
\end{equation}
Using $u(t,\cdot)= \int_0^t \pdt u(\eta,\cdot) \di\eta$ as $u(0,\cdot) =0$, we have that 
\begin{equation}\label{proof:uniq:est_h}
\abs{h(t)} \le  C_1 \int_0^t \vnorma{\pdt u(\eta)} \di \eta.
\end{equation}
Now, we take $\varphi=\pdt u(t)\in \hk{1}$ in \eqref{uni:var_for} and  integrate the result over $t\in(0,s)\subset (0,T)$ to get 
\begin{multline*} \label{001:eq:est2}
  \int_0^s \vnorma{\pdt u (t)}^2 \dt +  \D\int_0^s \int_\Omega  \eta |\nabla \pdt u|^2 \dX\dt  + \frac{1}{2}\D\int_0^s \int_\Omega  \kappa \pdt \abs{\nabla u}^2 \dX \dt \\
  =\int_0^s h(t) \scal{p(t)}{\pdt u(t)} \dt + \D\int_0^s \scal{F(u_1(t))-F(u_2(t))}{\pdt u(t)}\dt.
\end{multline*}
The third term on the left-hand side of this equation can be handled by using 
\begin{equation*}
 \int_0^s   \int_\Omega  \kappa\pdt \abs{\nabla u}^2 \dX \dt = \int_\Omega  \kappa(s) \abs{\nabla u(s)}^2 \dX -   \int_0^s   \int_\Omega (\pdt \kappa) \abs{\nabla u}^2 \dX \dt. 
\end{equation*}
Next, we focus on the terms on the right-hand side. For the first term, using the $\veps$-Young inequality, we obtain that
\begin{align*}
     \abs{\int_0^s h(t) \scal{p(t)}{\pdt u(t)} \dt } &\le    \veps \int_0^s \vnorma{\pdt u(t)}^2 \dt + \frac{\vnorma{p}_{\mathcal{X}}^2}{4\veps} \int_0^s \abs{h(t)}^2\dt \\
    & \stackrel{\eqref{proof:uniq:est_h}}{\le} \veps\int_0^s \vnorma{\pdt u(t)}^2 \dt+  \frac{\vnorma{p}_{\mathcal{X}}^2 C_1^2 T}{4\veps} \int_0^s \int_0^t \vnorma{\pdt u(\eta)}^2 \di\eta \dt .
   \end{align*}
Similarly, using the Lipschitz continuity of $F$ and $u(t,\cdot)= \int_0^t \pdt u(\eta,\cdot) \di\eta$, we deduce that 
\begin{multline*}
    \abs{\int_0^s \scal{F(u_1(t))-F(u_2(t))}{\pdt u(t)}\dt} \\ 
 %  & \le  \veps_1 \int_0^s \vnorma{\pdt u(t)}^2 \dt+\frac{L_F^2}{4\veps_1} \int_0^s \vnorma{ u(t)}^2 \dt \\
     \le   \veps \int_0^s \vnorma{\pdt u(t)}^2 \dt+\frac{L_F^2T}{4\veps} \int_0^s \int_0^t \vnorma{ \pdt u(\eta)}^2 \di\eta \dt. 
\end{multline*}
Summarising,  using \ref{as:DP:eta}-\ref{as:DP:kappa} and taking $\veps=\frac{1}{4}$, we obtain the estimate 
\begin{multline*}
      \int_0^s \vnorma{\pdt u(t)}^2 \dt + 2\ubar{\eta}_0 \int_0^s \vnorma{\nabla \pdt u(t)}^2 \dt  + \ubar{\kappa}_0\vnorma{\nabla u(s)}^2  \\
      \le  \ubar{\kappa}^\prime_1 \int_0^s \vnorma{\nabla u(t)}^2 \dt+C_2\int_0^s \int_0^t \vnorma{ \pdt u(\eta)}^2 \di\eta \dt,
\end{multline*}
where 
\[
        C_2 := 2 T \left(\vnorma{p}_{\mathcal{X}}^2 C_1^2+L_F^2\right). 
\]
Therefore, applying the Gr\"onwall argument gives that $u=0$ a.e.\ in $Q_T.$ Moreover, from 
\eqref{proof:uniq:est_h1}, we obtain that $h=0$ a.e. in $\Iopen.$
\end{proof}

%%%%%%%%%%%
\section{Existence of a solution}
\label{sec:existence}
%%%%%%%%%%%

In this section, we will show the existence of a weak solution to problem (\ref{eq:expression_h}-\ref{eq:var_for}) by employing Rothe's method. We start by dividing the time interval $[0, T]$  into $n \in \mathbb{N}$ equidistant subintervals $[t_{i-1},t_i]$ of length $\tau = T/n$, $i = 1,\ldots n$. Hence, $t_i = i \tau$ for $i = 0,1, \ldots,n$.  We consider for any function $z$ that
%$$
\[
z_i \approx z(t_i) \quad \text{ and } \quad \pdt z(t_i) \approx \delta z_i = \dfrac {z_i-z_{i-1}}{\tau},
\]
i.e. the backward Euler method is used to approximate the time derivatives at every time step $t_i$. Moreover, linearising the term containing $F$ in the right-hand side of \eqref{eq:var_for} at time step $t_i$ by replacing $u_i$ with $u_{i-1}$, we get the following time-discrete problem at time $t=t_i$: 
\medskip
\begin{center}
Find $u_i \in \hk{1}$ and $h_i\in \RR$ such that 
\begin{multline} \label{eq:disc_inv_prob}
\scal{\delta  u_i}{\varphi} +  \scal{\eta_i \nabla \delta u_i}{\nabla \varphi}  +  \scal{\kappa_i \nabla u_i}{\nabla \varphi} \\
= h_{i}\scal{p_i}{\varphi}+\scal{F(u_{i-1})}{\varphi} +   \scal{\eta_i (\pdt G)_i}{\varphi}_{\partial\Omega} +  \scal{g_i}{\varphi}_{\partial\Omega}, \quad \forall\varphi \in \hk{1}, 
\end{multline}
and
\begin{equation} \label{disc:hi-1}
h_{i}=\frac{1}{\omega_i}\left[(m^\prime)_i - \int_{\partial\Omega} \eta_i(\partial_t G)_i \dg - \int_{\partial\Omega} g_i \dg   - \int_\Omega F(u_{i-1}) \dX\right], 
\end{equation}
where
\begin{equation} \label{initC:disc_inv_prob}
 u_{0}=\tilde{u}_0. 
\end{equation}
\end{center}
\medskip

Hence, for any $i\in \left\{1,...,n\right\}$, we first derive $h_i \in \RR$ from \eqref{disc:hi-1} and we afterwards solve the following problem for $u_i$:
\begin{equation}\label{equiv:var_for_inv_disc_prob}
a_i(u_i,\varphi) = l_i(\varphi), \quad \forall \varphi \in \hk{1}, 
\end{equation}
where the bilinear form  $a_i: \hk{1}\times \hk{1}\to \mathbb{R}$ is given by 
\begin{equation*}
    a_i(u,\varphi) := \frac{1}{\tau}\scal{ u}{\varphi} +  \frac{1}{\tau} \scal{\eta_i \nabla u}{\nabla \varphi} +  \scal{\kappa_i \nabla u}{\nabla \varphi} 
\end{equation*}
and the linear functional $l_i: \hk{1}\to \mathbb{R}$ is defined by
\begin{multline*}
     l_i(\varphi) := h_{i}\scal{p_i}{\varphi}+\scal{F(u_{i-1})}{\varphi} +   \scal{ \eta_i(\pdt G)_i}{\varphi}_{\partial\Omega} +  \scal{g_i}{\varphi}_{\partial\Omega} \\
     + \frac{1}{\tau}\scal{ u_{i-1}}{\varphi} + \frac{1}{\tau} \scal{\eta_i \nabla u_{i-1}}{\nabla \varphi} .
\end{multline*}

The existence and uniqueness of $u_i$ are addressed in the following theorem. 

\begin{theorem}
   Let the conditions \ref{as:DP:eta} until \ref{as:m} be fulfilled.  Then, for any $i=1,\ldots,n$, there exists a unique couple $\{h_i, u_i\}\in \RR \times \hk{1}$ solving (\ref{eq:disc_inv_prob}-\ref{disc:hi-1}). 
\end{theorem}

\begin{proof}
Note that $a$ is bounded on $\hk{1}\times\hk{1}$, and $l_i$ is bounded on $\hk{1}$ if $u_{i-1}\in\hk{1}$. Moreover, the bilinear form $a_i$ is  $\hk{1}$-elliptic as
\[
a_i(u,u) \ge \min\left\{\frac{1}{\tau},\frac{\ubar{\eta}_0}{\tau}+ \ubar{\kappa}_0\right\} \vnorma{u}_{\hk{1}}^2, \quad \forall u\in \hk{1}. 
\]
Hence, starting from $u_0=\tilde{u}_0\in \hk{1}$, we recursively obtain (as the conditions of the Lax-Milgram lemma are satisfied) the existence and uniqueness of $h_i\in\RR$ and $u_i \in \hk{1}$ for $i=1,\ldots,n.$
\end{proof}

Now, we derive the a priori estimates for the discrete solutions to the inverse problem.  

\begin{lemma}\label{inv_prob:_est1}
Let the assumptions \ref{as:DP:eta} until \ref{as:m} be fulfilled. Then, there exists positive constants $C$ and $\tau_0$ such that 
\begin{equation}\label{est:invprob_discrsolu}
\max\limits_{1\le j\le n}\vnorma{u_j}_{\hk{1}}^2
+ \sum_{i=1}^n \vnorma{\delta u_i}_{\hk{1}}^2 \tau +  \sum_{i=1}^n \vnorma{u_i - u_{i-1}}_{\hk{1}}^2+ \sum_{i=1}^n |h_i|^2\tau \le C, 
\end{equation}
for any $\tau < \tau_0.$
\end{lemma}

\begin{proof}
Setting $\varphi = \delta u_i \tau$ in \eqref{eq:disc_inv_prob} and summing up the result for $i=1,\ldots,j$ with $1\le j\le n$ give
\begin{multline}\label{direct_problem:a_priori_estimate:eq1}
\sum_{i=1}^j \vnorma{\delta  u_i}^2 \tau +  \sum_{i=1}^j \scal{\eta_i \nabla\delta u_i}{\nabla  \delta  u_i} \tau  +  \sum_{i=1}^j \scal{\kappa_i \nabla u_i}{\nabla  \delta  u_i} \tau 
= \sum_{i=1}^j h_{i}\scal{p_i}{\delta u_i} \tau  \\
+ \sum_{i=1}^j \scal{F(u_{i-1})}{\delta u_i} \tau 
+ \sum_{i=1}^j  \scal{\eta_i (\pdt G)_i}{\delta u_i}_{\partial\Omega} \tau + \sum_{i=1}^j \scal{g_i}{\delta u_i}_{\partial\Omega} \tau.
\end{multline}
By \ref{as:DP:eta}, we have that
 \[
  \sum_{i=1}^j \scal{\eta_i \nabla\delta u_i}{\nabla  \delta  u_i} \tau \geq \ubar{\eta}_0 \sum_{i=1}^j \vnorma{ \nabla\delta u_i}^2 \tau.
 \]
About the left-hand side of \eqref{direct_problem:a_priori_estimate:eq1}, we note that
 \begin{multline*}
  \sum_{i=1}^j  \scal{\kappa_i \nabla u_i}{\nabla \delta u_i} \tau 
    = \half \scal{\kappa_j \nabla u_j}{\nabla u_j}
    - \half \scal{\kappa_0 \nabla \tilde{u}_0}{\nabla  \tilde{u}_0}\\
   - \half \sum_{i=1}^j \scal{(\delta\kappa_i) \nabla u_{i-1}}{\nabla u_{i-1}} \tau 
    + \half \sum_{i=1}^j \scal{\kappa_i (\nabla u_i -  \nabla u_{i-1})}{\nabla u_i - \nabla u_{i-1}}.
 \end{multline*}
 Hence, using \ref{as:DP:kappa}, we get for $\tau <1$ that
\begin{multline*}
  \sum_{i=1}^j  \scal{\kappa_i \nabla u_i}{\nabla \delta u_i} \tau
  \ge \frac{\ubar{\kappa}_0}{2}  \vnorma{\nabla u_j}^2 - \left( \frac{\ubar{\kappa}_1}{2} +  \frac{\ubar{\kappa}_1^\prime}{2}\right) \vnorma{\nabla \tilde{u}_0}^2 \\
							   - \frac{\ubar{\kappa}_1^\prime}{2}  \sum_{i=1}^{j-1} \vnorma{\nabla u_{i}}^2 \tau  
                        + \frac{\ubar{\kappa}_0}{2} \sum_{i=1}^j \vnorma{\nabla u_i - \nabla u_{i-1}}^2.  
\end{multline*}
Now, we will estimate the terms on the right-hand side of \eqref{direct_problem:a_priori_estimate:eq1}. 
Before doing this, using \eqref{eq:inequality_nonlinear_f}, we estimate $h_i$ given by \eqref{disc:hi-1} as follows
\[
\abs{h_i} \le H_i + \frac{{L}_F }{\ubar{\omega}_0}  \vnorma{u_{i-1}}_{\Leb^1(\Omega)}\le  H_i + C_1\vnorma{u_{i-1}}, \ \ C_1:=\frac{L_F\sqrt{\meas{\Omega}}}{\ubar{\omega}_0},
\]
with
\[
H_i := \frac{1}{\ubar{\omega}_0} \left[ \abs{ (m^\prime)_i } + \ubar{\eta}_1 \sqrt{\meas{\partial\Omega}} \vnorma{ (\partial_t G)_i}_{\partial\Omega} + \sqrt{\meas{\partial\Omega}} \vnorma{ g_i}_{\partial\Omega} + \abs{F(0)} \meas{\Omega}  \right].
\]
Please note that 
\begin{multline*}
\sum_{i=1}^j H_i^2 \tau\\ \le C_2:= C\left(\vnorma{m^\prime}_{\Leb^2 \Iopen},\vnorma{\pdt G}_{\lpkIX{2}{\Leb^2(\partial\Omega)}}, \vnorma{g}_{\lpkIX{2}{\Leb^2(\partial\Omega)}}, \meas{\overline{\Omega}},T, F(0), \ubar{\omega}_0, \ubar{\eta}_1 \right).    
\end{multline*}
Hence, using $u_{i-1} = \sum_{k=1}^{i-1} \delta u_k \tau+ \tilde{u}_0$ for $i\ge 1,$ we have that
\begin{equation} \label{eq:expression_ui}
\sum_{i=1}^j \vnorma{u_{i-1}}^2 \tau \le 2 T \vnorma{\tilde{u}_0}^2  + 2 T \sum_{i=1}^j \left(\sum_{k=1}^{i-1} \vnorma{\delta u_{k}}^2 \tau \right)\tau,
\end{equation}
and so
\begin{equation}\label{eq:discrete_estimate_hi}
\sum_{i=1}^j \abs{h_i}^2 \tau \le 2 C_2 + 2 C_1^2 \sum_{i=1}^j \vnorma{u_{i-1}}^2 \tau \le C_3 + 4 C_1^2 T  \sum_{i=1}^j \left(\sum_{k=1}^{i-1} \vnorma{\delta u_{k}}^2 \tau  \right) \tau,
\end{equation}
with $C_3:= 2 C_2 + 4 C_1^2 T \vnorma{\tilde{u}_0}^2.$ Therefore, employing the $\veps$-Young inequality, we obtain 
\[
\left|\sum_{i=1}^j h_{i} \scal{p_i}{\delta  u_i}\tau\right|
     \le \veps_1 \sum_{i=1}^j  \vnorma{\delta u_i}^2 \tau + \frac{\vnorma{p}_{{\mathcal X}}^2 }{4\veps_1} \left(C_3 + 4 C_1^2 T  \sum_{i=1}^j \left(\sum_{k=1}^{i-1} \vnorma{\delta u_{k}}^2 \tau  \right) \tau\right).
  % \label{est:h}
\]
For the second term on the right-hand side of \eqref{direct_problem:a_priori_estimate:eq1}, we use \eqref{eq:inequality_nonlinear_f} and \eqref{eq:expression_ui} to get
\begin{align*}
 \left| \sum_{i=1}^j \scal{F(u_{i-1})}{\delta u_i} \tau \right| &\le \veps_1 \sum_{i=1}^j  \vnorma{\delta u_i}^2 \tau + \frac{\ubar{L}_F}{4\veps_1} \sum_{i=1}^j \left(\vnorma{u_{i-1}}^2+1\right) \tau \\
 &\le C_4(\veps_1) +  \veps_1 \sum_{i=1}^j  \vnorma{\delta u_i}^2 \tau  +  \frac{\ubar{L}_FT}{2\veps_1} \sum_{i=1}^j \left(\sum_{k=1}^{i-1} \vnorma{\delta u_{k}}^2 \tau  \right) \tau, 
\end{align*}
with $\ubar{L}_F:= 2 \max \left\{L_F^2, F(0)^2 \abs{\Omega} \right\}$ and $C_4(\veps):=\frac{\ubar{L}_FT}{4\veps} \left(1+ 2\vnorma{\tilde{u}_0}^2 \right)$.
Next, using the trace theorem ($\vnorma{\phi}_{\Leb^2(\Gamma)} \le C_{\textrm{tr}} \vnorma{\phi}_{\hk{1}}$), we get that 
\begin{align*}
\abs{\sum_{i=1}^j  \scal{ \eta_i(\pdt G)_i}{\delta u_i}_{\partial\Omega} \tau} &\le \frac{\ubar{\eta}_1^2}{4\veps_2} \sum_{i=1}^j \vnorma{(\pdt G)_i}_{\Leb^2(\Gamma)}^2 \tau + \veps_2 \sum_{i=1}^j  \vnorma{\delta u_i}_{\Leb^2(\Gamma)}^2 \tau \\
& \le \frac{\ubar{\eta}_1^2}{4\veps_2} C\left(\vnorma{\pdt G}_{\lpkIX{2}{\Leb^2(\partial \Omega)}}\right) + \veps
_2  C_{\textrm{tr}}^2 \sum_{i=1}^j  \vnorma{\delta u_i}_{\hk{1}}^2 \tau. 
\end{align*}
Similarly, we have that
\[
\abs{\sum_{i=1}^j \scal{g_i}{\delta u_i}_{\partial\Omega} \tau} 
\le \frac{1}{4\veps_2} C\left(\vnorma{g}_{\lpkIX{2}{\Leb^2(\partial \Omega)}}\right) + \veps
_2  C_{\textrm{tr}}^2 \sum_{i=1}^j  \vnorma{\delta u_i}_{\hk{1}}^2 \tau. 
\]
Now, we collect all estimates derived above and obtain from \eqref{direct_problem:a_priori_estimate:eq1} that 
\begin{multline}\label{ener:ineq22}
  \left(1- 2\veps_1 - 2 C_{\textrm{tr}}^2 \veps_2 \right) \sum_{i=1}^j \vnorma{\delta u_i}^2 \tau + \left(\ubar{\eta}_0 - 2 C_{\textrm{tr}}^2 \veps_2\right) \sum_{i=1}^j \vnorma{ \nabla \delta  u_i}^2 \tau \\
    + \frac{\ubar{\kappa}_0}{2}  \vnorma{\nabla u_j}^2 +  \frac{\ubar{\kappa}_0}{2}  \sum_{i=1}^j \vnorma{\nabla u_i - \nabla u_{i-1}}^2 \\ 
    \le C_5 +  \frac{\ubar{\kappa}_1^\prime}{2}  \sum_{i=1}^{j-1} \vnorma{\nabla u_{i}}^2 \tau   + C_6\sum_{i=1}^{j}\left(\sum_{k=1}^{i} \vnorma{\delta u_k }^2\tau\right)\tau,
\end{multline}
where
\begin{align*}
  C_5 &:=\left( \frac{\ubar{\kappa}_1}{2} +  \frac{\ubar{\kappa}_1^\prime}{2}\right) \vnorma{\nabla \tilde{u}_0}^2  + \frac{\vnorma{p}_{{\mathcal X}}^2 }{4\veps_1} C_3 + C_4(\veps_1) + \frac{\ubar{\eta}_1^2}{4\veps_2} C\left(\vnorma{\pdt G}\right) +  \frac{1}{4\veps_2} C\left(\vnorma{g}\right),   \\
  C_6 &:=  \frac{\vnorma{p}_{{\mathcal X}}^2 C_1^2 T}{\veps_1} +  \frac{\ubar{L}_FT}{2\veps_1} .
\end{align*}
First, we take $\veps_1$ and $\veps_2$ small enough such that $1- 2\veps_1 - 2 C_{\textrm{tr}}^2 \veps_2>0$ and $\ubar{\eta}_0 - 2 C_{\textrm{tr}}^2 \veps_2>0.$ Then, we apply the Gr\"onwall lemma to obtain the existence of positive constants $C$ and $\tau_0$ such that 
\[
\sum_{i=1}^j \vnorma{\delta u_i}^2 \tau +\sum_{i=1}^j \vnorma{\nabla \delta u_i}^2 \tau 
    +\vnorma{\nabla u_j}^2 +  \sum_{i=1}^j \vnorma{\nabla u_i - \nabla u_{i-1}}^2 \le C,
\]
for $\tau < \tau_0.$ Moreover, from \eqref{eq:discrete_estimate_hi}, for $\tau < \tau_0$, we obtain the existence of a positive constant $C$ such that 
\[
\sum_{i=1}^j \abs{h_i}^2 \tau \le C. \qedhere
\]
\end{proof}

In the next step, we introduce the so-called Rothe functions: The piecewise linear-in-time function
\begin{equation}
    \label{Rothestepfun:u} 
    U_n:[0,T]\to\lp{2}:t\mapsto\begin{cases} \tilde{u}_0 & t =0,\\
u_{i-1} + (t-t_{i-1})\delta u_i &
     t\in (t_{i-1},t_i],\quad 1\leqslant i\leqslant n, \end{cases}\end{equation}
and the piecewise constant function
\begin{equation}
    \label{Rothefun:u}
\overline U_n:[-\tau,T] \to\lp{2}:t\mapsto\begin{cases} \tilde{u}_0 &  t \in [-\tau,0],\\
 u_i &     t\in (t_{i-1},t_i],\quad 1\leqslant i\leqslant n.\end{cases}
 \end{equation}
Similarly, in connection with the given functions $\eta$, $\kappa$, $g$, $\pdt G$, $\omega$, and $p,$ we define the functions $\overline{\eta}_n$, $\overline{\kappa}_n$, $\overline{g}_n$, $\overline{\pdt G}_n$, $\overline{\omega}_n$ and $\overline{p}_n$, respectively. 
Now, using these Rothe functions,
 we rewrite the discrete variational formulation (\ref{eq:disc_inv_prob}-\ref{disc:hi-1}) as follows (for all $t\in(0,T]$)
\begin{multline}\label{eq:disc_inv_prob:whole_time_frame}
    \scal{ \pdt U_n(t) }{\varphi} + \scal{ \overline{\eta}_n(t) \nabla \pdt U_n(t)}{\nabla \varphi}  +  \scal{\overline{\kappa}_n(t) \nabla \overline U_n(t)}{\nabla \varphi} 
= \overline{h}_n(t)\scal{\overline{p}_n(t)}{\varphi} \\
+\scal{F(\overline{U}_n(t-\tau))}{\varphi} 
+  \scal{\overline{\eta}_n(t)\overline{\pdt G}_n(t)}{\varphi}_{\partial\Omega} +  \scal{\overline{g}_n(t)}{\varphi}_{\partial\Omega}, \quad \forall\varphi \in \hk{1},
\end{multline}
and
\begin{equation} \label{disc:hi-1:whole_time_frame}
\overline{h}_n(t)=\frac{1}{\overline{\omega}_{n}(t)}\left[ {\overline{m^\prime}_n(t)} - \scal{\overline{\eta}_n(t) \overline{\pdt G}_n(t)}{1}_{\partial\Omega}- \scal{ \overline{g}_n(t)} {1}_{\partial\Omega}   - \scal{F(\overline{u}_{n}(t-\tau))}{1}\right].
\end{equation}

Now, we are ready to show the existence of a solution to (\ref{eq:expression_h}-\ref{eq:var_for}).

%%%%%%%%%%%%%%%
\begin{theorem}\label{thm:existence_inverse_problem}
Let the assumptions \ref{as:DP:eta} until \ref{as:m} be fulfilled.  Then, there exists a unique weak solution couple $\{u,h\}$ to (\ref{eq:expression_h}-\ref{eq:var_for}) satisfying 
\begin{equation*}
 u \in \cIX{\hk{1}} \quad \text{with} \quad  \pdt u \in \lpkIX{2}{\hk{1}} \quad \text{and} \quad h \in \Leb^2\Iopen.
\end{equation*}
\end{theorem}
%%%%%%%%%%%%%%%

\begin{proof}
We have from \Cref{inv_prob:_est1} that there exist $C>0$ and $n_0\in \NN$ such that for all $n\geqslant n_0 > 0$ it holds that 
\begin{multline}\label{est:invprob_discr:solu}
\sup\limits_{t\in[0,T]} \vnorma{\overline{U}_{n}(t)}_{\hk{1}}^2 
+ \int\limits_0^T \vnorma{\pdt{U}_{n}(t)}_{\hk{1}}^2  \dt \\
+ \D \sum_{i=1}^n \vnorma{\int_{t_{i-1}}^{t_i} \pdt{U}_{n}(s)\ds}_{\hk{1}}^2+ \vnorma{\overline{h}_{n}}_{\Leb^2(0,T)}^2\le C.
\end{multline}
By \cite[Lemma~1.3.13]{Kacur1985}, the compact embedding  $\hk{1} \imbed \imbed \lp{2} $ (see \cite[Theorem~6.6-3]{Ciarlet2013}) leads to the existence of a function
$u \in \cIX{\lp{2}}  \cap \Leb^{\infty}\left((0,T), \hk{1}\right)$ with $\pdt u \in \lpkIX{2}{\lp{2}}$,
and a subsequence $\{ U_{n_l}\}_{l\in\NN}$ of $\{U_n\}$ such that
\begin{equation} \label{convergence:rothe_functions_dp}
\left\{
\begin{array}{ll}
U_{n_l} \to u & \text{in}~~\Cont\left([0,T], \lp{2}\right) , \\[4pt]
U_{n_l}(t) \rightharpoonup u(t) & \text{in}~~\hk{1},~~\forall t \in [0,T], \\[4pt]
\overline{U}_{n_l}(t) \rightharpoonup u(t) & \text{in}~~\hk{1},~~\forall t \in [0,T],  \\[4pt]
\pdt {U}_{n_l} \rightharpoonup \pdt u & \text{in}~~\Leb^{2}\left((0,T), \lp{2}\right).
\end{array}
\right.
\end{equation}
By the reflexivity of the space $\lpkIX{2}{\hk{1}}$, we have
%the existence of a subsequence of ${U}_{n_l}$ (denoted by the same symbol)  such that 
that
\begin{equation}\label{weak_convergence_time_der}
\pdt {U}_{n_l} \rightharpoonup \pdt u \quad  \text{in} \quad \lpkIX{2}{\hk{1}}, 
\end{equation}
i.e. $u\in\cIX{\hk{1}}$. Similarly, we have that 
\begin{equation}\label{weak_convergence_hn}
\overline{h}_{n_l} \rightharpoonup \sigma \quad \text{in} \quad \Leb^2\Iopen.
\end{equation}
Moreover, from \Cref{inv_prob:_est1}, we also have that (note that $\tau_l = T/n_l$)
\begin{equation}\label{dir:str:conv:L2}
\int_0^{T} \left( \vnorma{\overline{U}_{n_l}(t) - U_{n_l}(t) }^2 + \vnorma{\overline{U}_{n_l}(t-\tau) - U_{n_l}(t) }^2 \right) \dt  \leqslant 2 \tau_l^2 \sum_{i=1}^{n_l} \vnorma{\delta u_i}^2 \tau  \leqslant C \tau_l^2,
\end{equation}
so \begin{equation}\label{dir:str:conv:L22}
\overline U_{n_l},  \overline U_{n_l}(\cdot-\tau) \to u \ \text{in} \ \lpkIX{2}{\lp{2}} \ \text{as} \ l\rightarrow \infty.
\end{equation}
Hence, using these limit transitions and the Lipschitz continuity of $F$, we easily see that 
\[
\int_0^T \overline{h}_{n_l}(t) \phi(t)\dt \rightarrow \int_0^T {h}(t) \phi(t)\dt \quad \text{ for all } \phi \in \Leb^2\Iopen \text { as } l \to \infty, 
\]
where $h$ is given by \eqref{eq:expression_h}. Hence, by the uniqueness of the weak limit, we have that $\sigma=h.$ Next, we integrate \eqref{eq:disc_inv_prob:whole_time_frame} for $n=n_l$ over $t\in(0,\eta)\subset \Iopen$ and pass to the limit $l\to \infty.$ Afterwards, we differentiate with respect to $\eta$ to obtain that \eqref{eq:var_for} is satisfied. The uniqueness of a solution and the convergence of the whole Rothe sequences follow from \Cref{thm:uniq_inv_problem}. 
\end{proof}

\begin{remark}[Direct problem for a semilinear pseudo-parabolic equation with Neumann boundary condition]
For the direct problem \eqref{eq:problem} with given source $f,$ the weak formulation becomes 
\begin{center}
Find $u(t)\in \hk{1}$ with $\pdt u(t)\in \hk{1}$ such that for a.a. $t \in \Iopen$ and any $\varphi \in \hk{1}$ it holds that 
\begin{multline}\label{eq:var_for_dp} 
\scal{ \pdt u(t)}{\varphi} + \scal{ \eta(t) \nabla \pdt u(t)}{\nabla \varphi} +  \scal{\kappa(t) \nabla u(t)}{\nabla \varphi} \\
=\scal{f(t)}{\varphi} + \scal{F(u(t))}{\varphi} + \scal{\eta(t)  \partial_t G(t)}{\varphi}_{\partial\Omega} + \scal{ g(t)}{\varphi}_{\partial\Omega}.    
\end{multline}
\end{center}
Putting $h=0$ in the approach followed for the inverse problem, we also get the existence of a unique weak solution to \eqref{eq:var_for_dp} satisfying 
\[
 u \in \cIX{\hk{1}} \quad \text{with} \quad  \pdt u \in \lpkIX{2}{\hk{1}},
\]
when the assumptions \ref{as:DP:eta} until \ref{as:g}, and $f\in \Leb^2\left(\I,\lp{2}\right)$ are satisfied.
\end{remark}

%%%%%%%%%%%%%%%%%%%%%%%%%%%%%%%%%%%%%%%%%%%%%%%%%%%%%%%%%%%%
\section{Numerical experiments}
\label{sec:experiments}
%%%%%%%%%%%%%%%%%%%%%%%%%%%%%%%%%%%%%%%%%%%%%%%%%%%%%%%%%%%%

In this section, we illustrate the behaviour of the proposed time-stepping
algorithm for recovering the unknown temporal source $h(t)$ in the inverse
problem (\ref{eq:problem}-\ref{eq:add:cond}).
We first present in \Cref{subsec:noisy_free} a convergence test for noise-free data in one dimension. Afterwards, in \Cref{subsec:1dnoisy,subsec:2dnoisy}, we perform 
noisy experiments in one and two spatial dimensions.
All computations were carried out with the FEniCSx platform
\cite{FEniCSx1,FEniCSx2,FEniCSx3}, based on version \texttt{0.1} of the
DOLFINx module.

The algorithm described in \Cref{algorithm} requires the time derivative $m^\prime(t)$ of the measurement. More specific, the computation proceeds by first
deriving $h_i$ from \eqref{disc:hi-1} requiring $(m^\prime)_i$ and $u_{i-1}$, and then solving the variational problem \eqref{equiv:var_for_inv_disc_prob}
for $u_i$. In the presence of noise, we stabilise this step by fitting the
noisy data $m^\epsilon(t)$ with a low-degree (denoted by $\tilde{d}$) polynomial $p^{\epsilon}_{\tilde{d}}(t)$, and differentiating 
the resulting polynomial analytically.
This avoids numerical differentiation of noisy data.

Let us first explain how we have constructed the approximating polynomials $p^{\epsilon}_{\tilde{d}}(t)$.  From $m(t)$, the noisy data $m^{\epsilon}(t)$ are generated by 
\[
m^{\epsilon}(t) = m(t) \left(1 + \epsilon {\mathcal R}(t)\right),
\qquad t \in \left\{ \frac{jT}{\widetilde{N}} : j=0,\ldots,\widetilde{N}\right\},
\]
where $\epsilon$ denotes the relative noise level (e.g.\ $\epsilon=0.01$
corresponds to $1\%$ noise), and $\mathcal{R}(t)$ is an independent random
variable uniformly distributed in $[-1,1]$. To ensure reproducibility, the random sequence is generated with a fixed seed
$\texttt{np.random.seed}(\texttt{exp}-1)$ for each experiment (\texttt{exp} is the number of the experiment). We have considered $\widetilde{N}=100$ in our experiments. Afterwards, for each noise level $\epsilon = \{0.001,0.005,0.01,0.03,0.05\}$, we have approximated (so regularised) $m^\epsilon(t)$ by a
polynomial $p^{\epsilon}_{\tilde{d}}(t)$ of degree $\tilde{d}$ using a linear least-squares fit (employing \texttt{np.polyfit}). The fitting degree of the polynomial has been selected on the relative improvement in discrete $\ell^2$ error
\begin{equation}\label{eq:definition_improvememt}
r_{\textup{im}}(\tilde{d}) = 100 \times \frac{\|p^\epsilon_{\tilde{d}-1} - m^\epsilon\|_{\ell^2} - \|p^\epsilon_{\tilde{d}} - m^\epsilon\|_{\ell^2}}{\|p^\epsilon_{\tilde{d}-1} - m^\epsilon\|_{\ell^2}} \%,
\end{equation}
where $\|f\|_{\ell^2} = \sqrt{\frac{1}{\widetilde{N}+1}\sum_{i=0}^{\widetilde{N}} \abs{f(t_i)}^2}$.
If $\epsilon \neq 0$, then the approximate value for $m^\prime(t_i)$ used in the reconstruction \eqref{disc:hi-1} has been  obtained by analytic
differentiation of this polynomial, i.e., $m^\prime(t_i) \approx  {(p^{\epsilon}_{\tilde{d}})}^{\prime}(t_i)$ for $i=1,\ldots,n$.

\begin{algorithm}[htbp]
\caption{Time-stepping algorithm for the discrete inverse problem}
\label{algorithm}
\begin{algorithmic}[1]
\Require initial field $u_0=\tilde u_0$, time step $\tau$, number of finite elements $n$
\For{$i = 1$ to $n$}
  \State {\bf compute } $h_i$ explicitly from \eqref{disc:hi-1} using $u_{i-1}$, and $m^\prime(t_i)$ if $\epsilon = 0$ or   $m^\prime(t_i) \approx  {(p^{\epsilon}_{\tilde{d}})}^{\prime}(t_i) $ if $\epsilon \neq 0$.
  \State {\bf assemble} the linear functional $l_i(\cdot)$ --see \eqref{equiv:var_for_inv_disc_prob}-- with $h_i$ and $u_{i-1}$.
  \State {\bf assemble} the bilinear form $a_i(\cdot,\cdot)$, see \eqref{equiv:var_for_inv_disc_prob}.
  \State {\bf finite element discretisation:} use P1--FEM on a mesh of $n$ elements to build the system matrix $A_i$ and right-hand side vector $b_i$:
    \[
      (A_i)_{jk} = a_i(\phi_k,\phi_j), \qquad (b_i)_j = l_i(\phi_j),
    \]
    where $\{\phi_j\}$ is the P1 basis.
  \State {\bf solve} the linear system $A_i U_i = b_i$ for $U_i$, and set $u_i := \sum_j (U_i)_j \phi_j$.
  \State {\bf store} $u_i, h_i$.
\EndFor
\Ensure discrete solution $\{u_i\}_{i=0}^n$ and scalars $\{h_i\}_{i=1}^n$.
\end{algorithmic}
\end{algorithm}

For all simulations, we take
\[
\eta=0.5,\qquad \kappa=t+1. 
\]
The corresponding discrete elliptic problems are solved numerically by applying the finite element method using the first-order (P1--FEM) Lagrange polynomials for the space discretisation. In 1D, the spatial interval $\Omega = (0,1)$ is partitioned into $200$ uniform elements. In 2D, we use a structured triangular mesh of size $40\times 40$ on the unit square $\Omega=(0,1)^2$. In all experiments, the final time is $T=1$ with $n=200$ time steps, i.e., $\tau = T/n = 0.005$.

%%%%%%%%%%%%%%%%%%%%%%%%%%%%%%%%%%%%%%%%%%%%%%%%%%%%%%%%%%%%%%
\subsection{Noise-free convergence test in 1D}
\label{subsec:noisy_free}
%%%%%%%%%%%%%%%

We begin with a one-dimensional convergence study ($\Omega = (0,1)$) using an exact measurement.  
The exact solution is chosen as
\[
  u_1(t,x) = \exp(-t) \sin(\pi x),
  \qquad
  h_1(t) = \exp(-t),
\]
for which
\[
  m_1(t) = \int_0^1 u(t,x)\,dx = \frac{2}{\pi}\exp(-t),
\]
and
\[
 p_1(x) = -\frac{\pi^2}{2}\sin(\pi x), \quad  F_1(u) = -u,
  \quad
  f_1(t,x) = \pi^2(t+1)\exp(-t)\sin(\pi x).
\]
We examine the order of convergence of $u$ and $h$ by evaluating the errors
\[
E_{\mathrm{max}}^{u}(\tau) =\max_{1\le i \le n}  \vnorma{u(t_i) - u_i} \quad \text{ and } \quad E_{\mathrm{max}}^{h}(\tau) = \max_{1\le i \le n} \abs{h(t_i) - h_i }.
\]
The computed results give the natural convergence rates $E_{\mathrm{max}}^{u_1}(\tau)=\OO{\tau}$ and $E_{\mathrm{max}}^{h_1}(\tau)= \OO{\tau},$ see \Cref{Experiment1conv}. 

%%%%%%%%%%%
\begin{figure}[htbp]
\begin{center}
\subfigure[]{\includegraphics[width=0.75\textwidth,angle=0,height = 0.22\textheight]{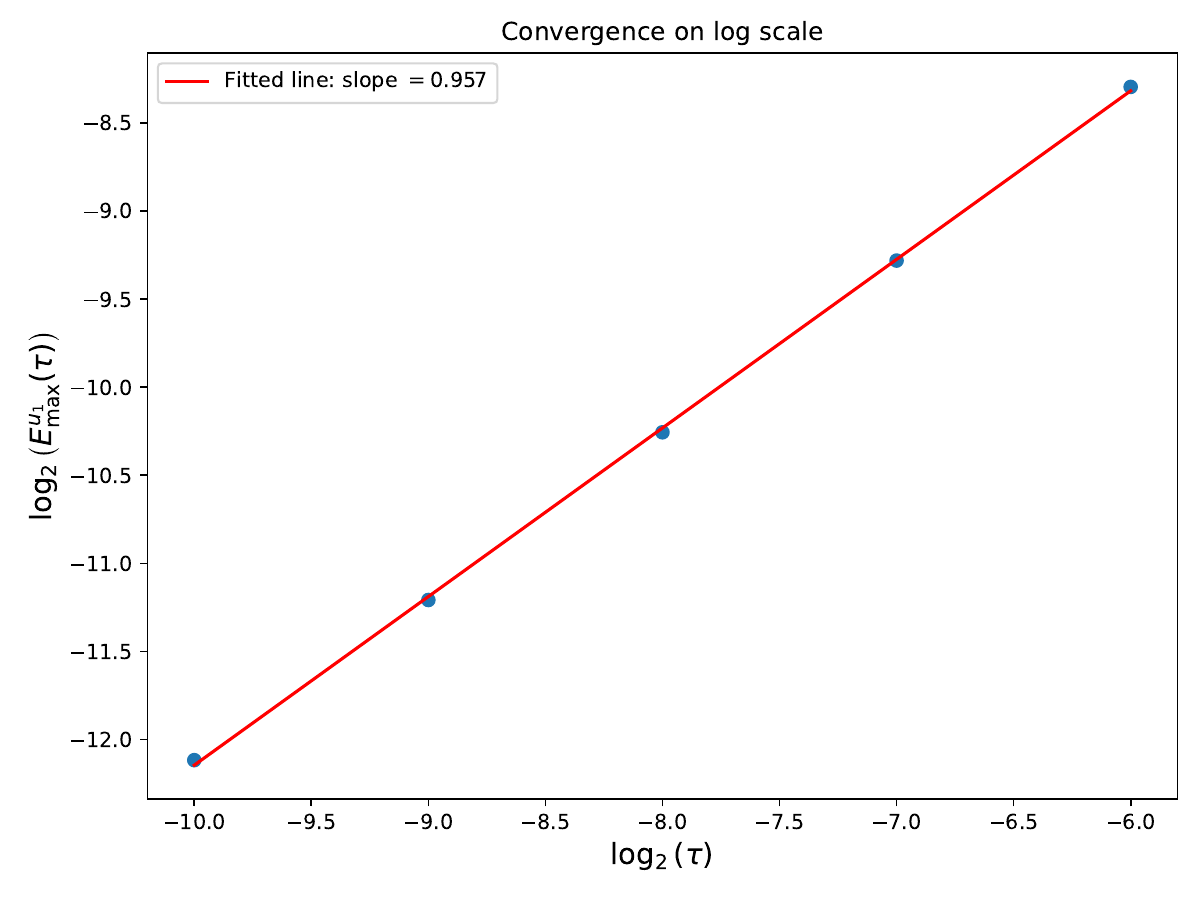}}
\subfigure[]{\includegraphics[width=0.75\textwidth,angle=0,height = 0.22\textheight]{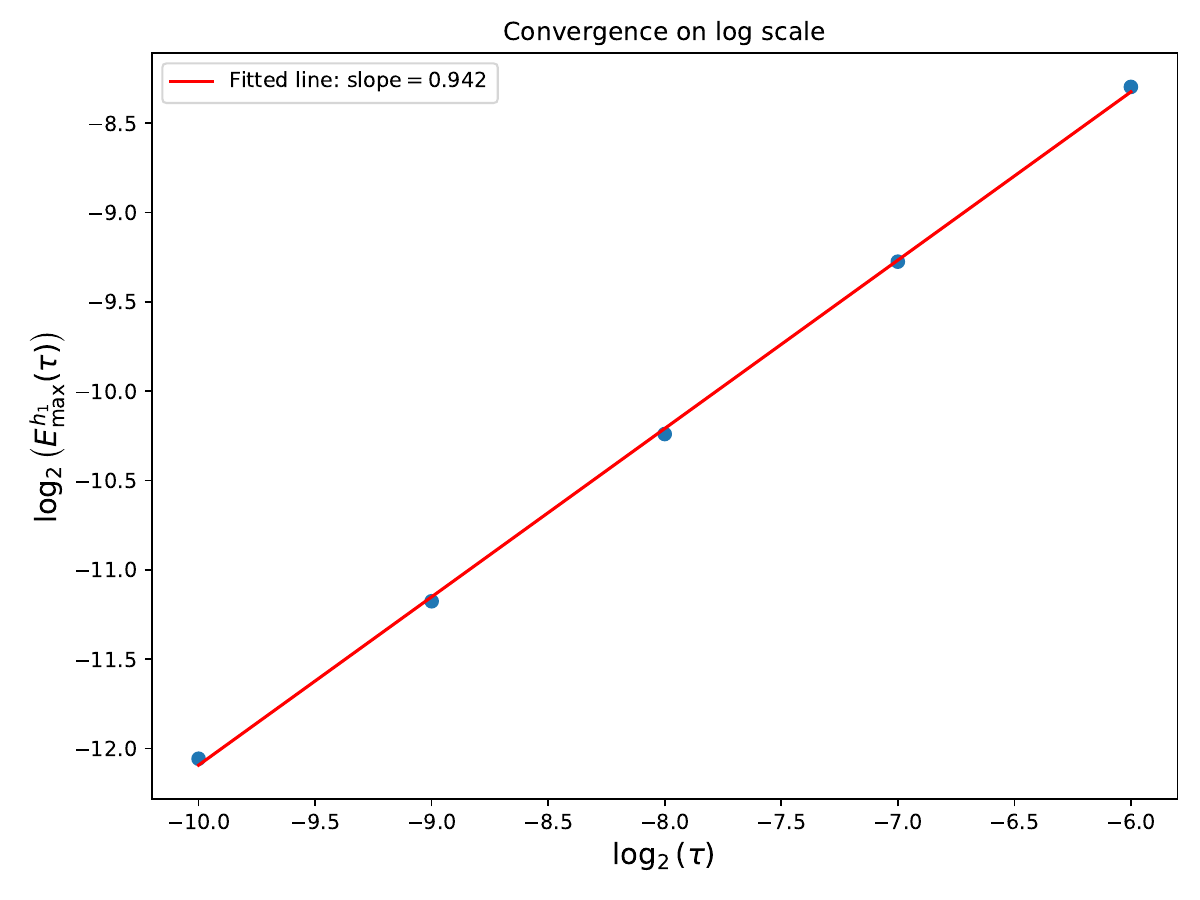}}
\end{center}
\vspace{-0.5cm}
\caption{Experiment 1 (1D): (a) rate of convergence of $u_1$; (b) rate of convergence for $h_1$, for noise-free data $m_1$. 
}
\label{Experiment1conv}
\end{figure}

%%%%%%%%%%%%%%%%%%%%%%%%%%%%%%%%%%%%%%%%%%%%%%%%%%%%%%%%%%%%%%%
\subsection{One-dimensional experiments with noisy data}
\label{subsec:1dnoisy}
%%%%%%%%%%%%%%%

In this section, we consider noisy experiments in one dimension ($\Omega = (0,1)$). In the first experiment, we use the same solution $\{u_1,h_1\}$ as in the convergence test. 
 We select the polynomial approximation degree for the measurement data $m_1^\epsilon(t)$ on basis of the relative improvement $r_{\text{im}}(\tilde{d})$ defined in \eqref{eq:definition_improvememt}. As shown in \Cref{tab:degree_selection_1D_1}, for $\epsilon \leq 0.01$, polynomials of degree~3 provide substantial improvements over lower degrees while higher degrees yield only small improvements. For higher noise levels ($\epsilon \geq 0.03$), the improvements beyond degree~2 become negligible. Based on this analysis, we have selected cubic polynomials  (i.e., $p^{\epsilon}_{3}(t)$) for this experiment to ensure adequate approximation quality across all different noise levels.  The polynomial fits for Experiment~1 with $\epsilon = 0.005$ are visualised in \Cref{fig:polynomial_analysis_exp1}, demonstrating how higher-degree polynomials (for $\tilde{d}\le 3$) better approximate the measurement data.  The results of the first experiment are depicted in \Cref{Experiment1,Experiment1u}. 

\begin{table}[htbp]
\centering
\caption{Polynomial approximation errors $\mathcal{E}(\tilde{d}) = \|p^\epsilon_{\tilde{d}} - m_1^\epsilon\|_{\ell^2}$ and relative improvements $r_{\text{im}}(\tilde{d}) = \frac{\mathcal{E}(\tilde{d}-1) - \mathcal{E}(\tilde{d})}{\mathcal{E}(\tilde{d}-1)}$ for Experiment 1 (1D case).}
\label{tab:degree_selection_1D_1}
\begin{tabular}{lcccccc}
\hline
$\tilde{d}$ &  & \multicolumn{5}{c}{Noise level $\epsilon$} \\
\cline{3-7}
& & 0.001 & 0.005 & 0.010 & 0.030 & 0.050 \\
\hline
1 & $\mathcal{E}(1)$ & 1.48e-2 & 1.49e-2 & 1.52e-2 & 1.71e-2 & 2.00e-2 \\
\hline
2 & $\mathcal{E}(2)$ & 1.26e-3 & 1.67e-3 & 2.57e-3 & 6.92e-3 & 1.14e-2 \\
  & $r_{\text{im}}(2)$ & +91.5\% & +88.8\% & +83.1\% & +59.6\% & +42.7\% \\
\hline
3 & $\mathcal{E}(3)$ & 2.36e-4 & 1.14e-3 & 2.28e-3 & 6.83e-3 & 1.14e-2 \\
  & $r_{\text{im}}(3)$ & +81.2\% & +31.7\% & +11.5\% & +1.3\% & +0.4\% \\
\hline
4 & $\mathcal{E}(4)$ & 2.28e-4 & 1.14e-3 & 2.28e-3 & 6.82e-3 & 1.14e-2 \\
  & $r_{\text{im}}(4)$ & +3.7\% & +0.0\% & +0.0\% & +0.1\% & +0.1\% \\
\hline
5 & $\mathcal{E}(5)$ & 2.27e-4 & 1.14e-3 & 2.27e-3 & 6.82e-3 & 1.14e-2 \\
  & $r_{\text{im}}(5)$ & +0.1\% & +0.0\% & +0.0\% & +0.0\% & +0.0\% \\
\hline
\end{tabular}
\end{table}
%%%%

\begin{figure}[htbp]
\begin{center}
\subfigure[]{\includegraphics[width=0.45\textwidth,angle=0,height = 0.22\textheight]{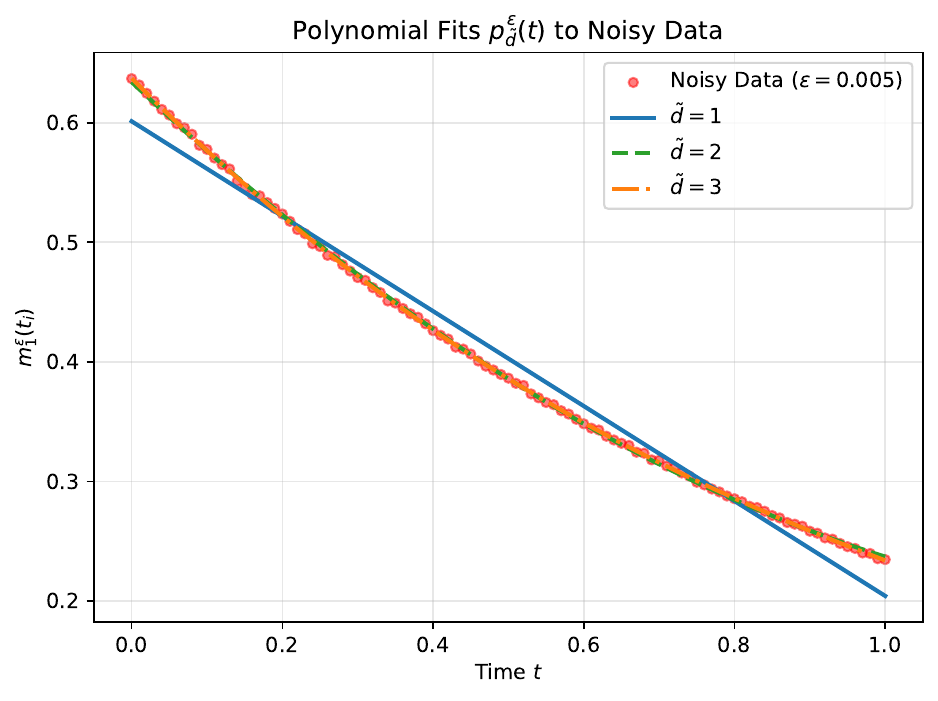}}
\subfigure[]{\includegraphics[width=0.45\textwidth,angle=0,height = 0.22\textheight]{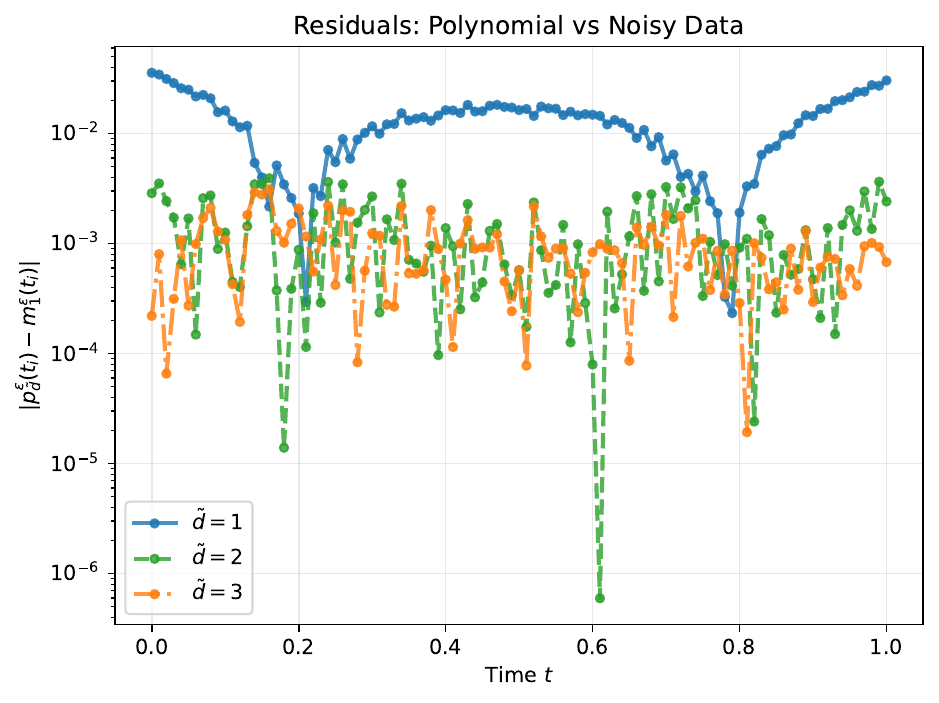}}
\end{center}
\vspace{-0.5cm}
\caption[Polynomial approximation analysis for Experiment 1]{Experiment 1 (1D): (a)~Polynomial approximations of noisy measurement data $m_1^\epsilon(t)$ for $\epsilon = 0.005$, and (b)~corresponding absolute errors.}
\label{fig:polynomial_analysis_exp1}
\end{figure}

%%%%%%%%%%%
\begin{figure}[htbp]
\begin{center}
\subfigure[]{\includegraphics[width=0.75\textwidth,angle=0,height = 0.22\textheight]{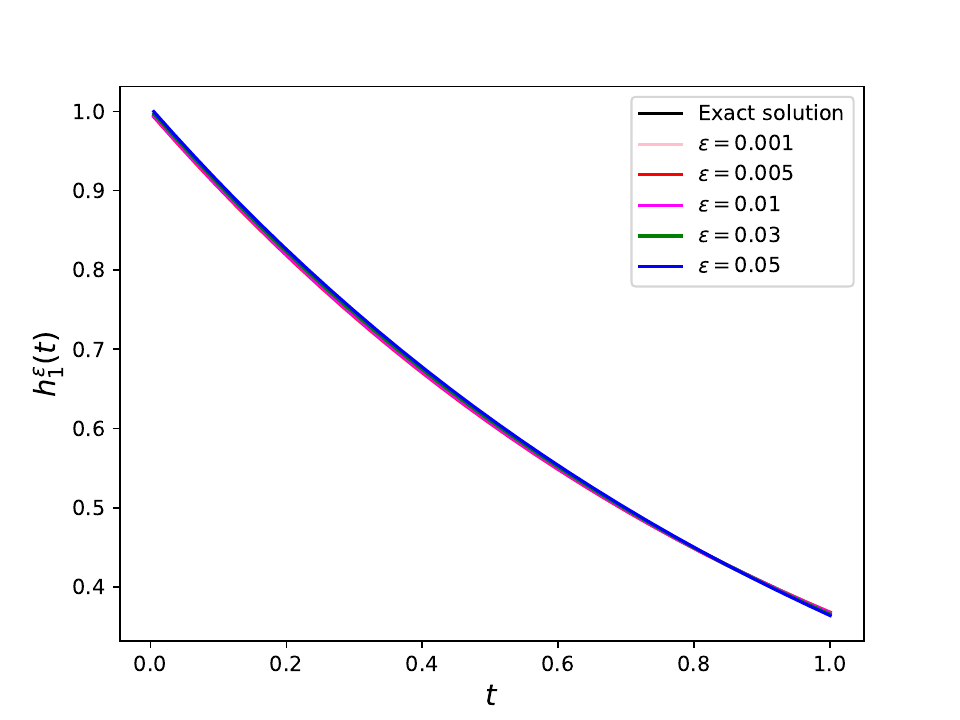}}
\subfigure[]{\includegraphics[width=0.75\textwidth,angle=0,height = 0.22\textheight]{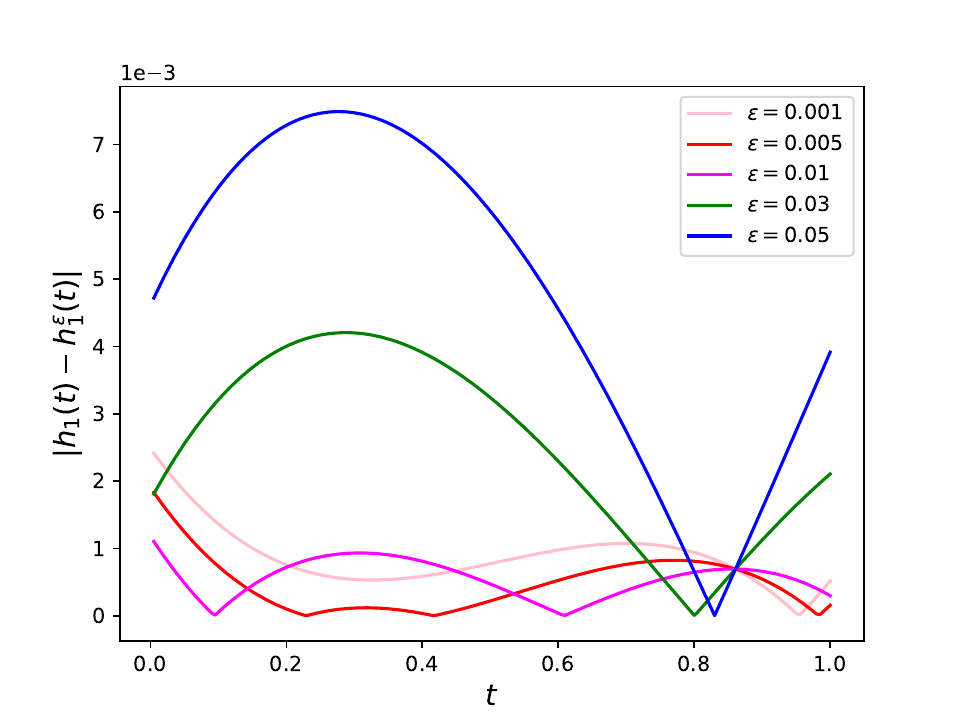}}
\end{center}
\vspace{-0.5cm}
\caption[Unknown time source (1D): Results for Experiment 1]{Experiment 1: (a)~The exact source and its numerical approximation, and (b)~its corresponding absolute error, obtained for various levels of noise.
}
\label{Experiment1}
\end{figure}
%%%%%%%%%%%

%%%%%%%%%%%
\begin{figure}[htbp]
\begin{center}
\subfigure[]{\includegraphics[width=0.75\textwidth,angle=0,height = 0.22\textheight]{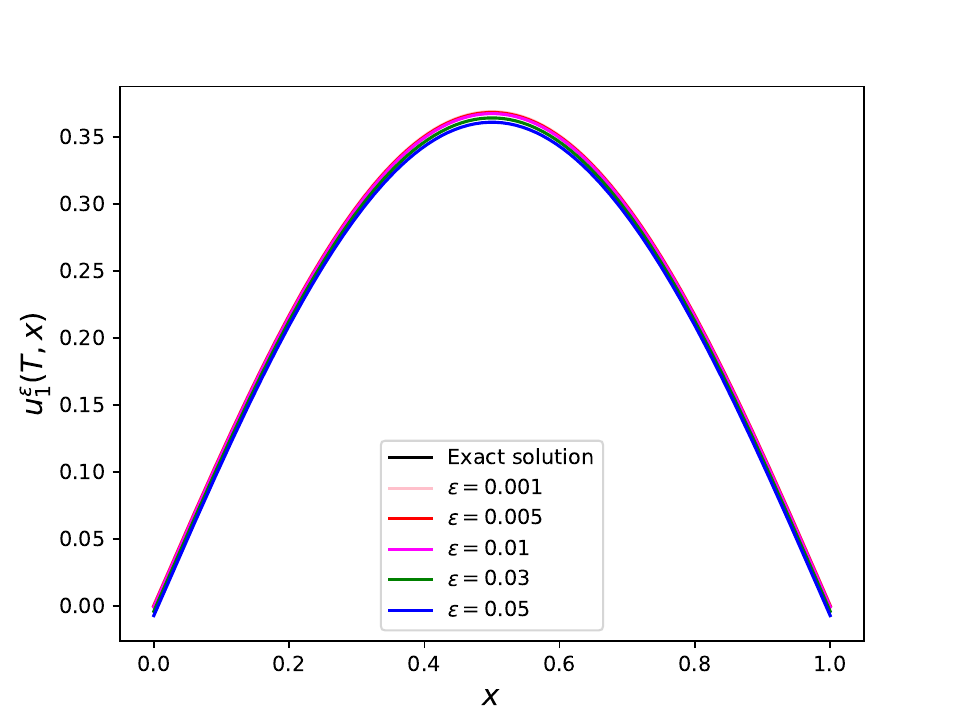}}
\subfigure[]{\includegraphics[width=0.75\textwidth,angle=0,height = 0.22\textheight]{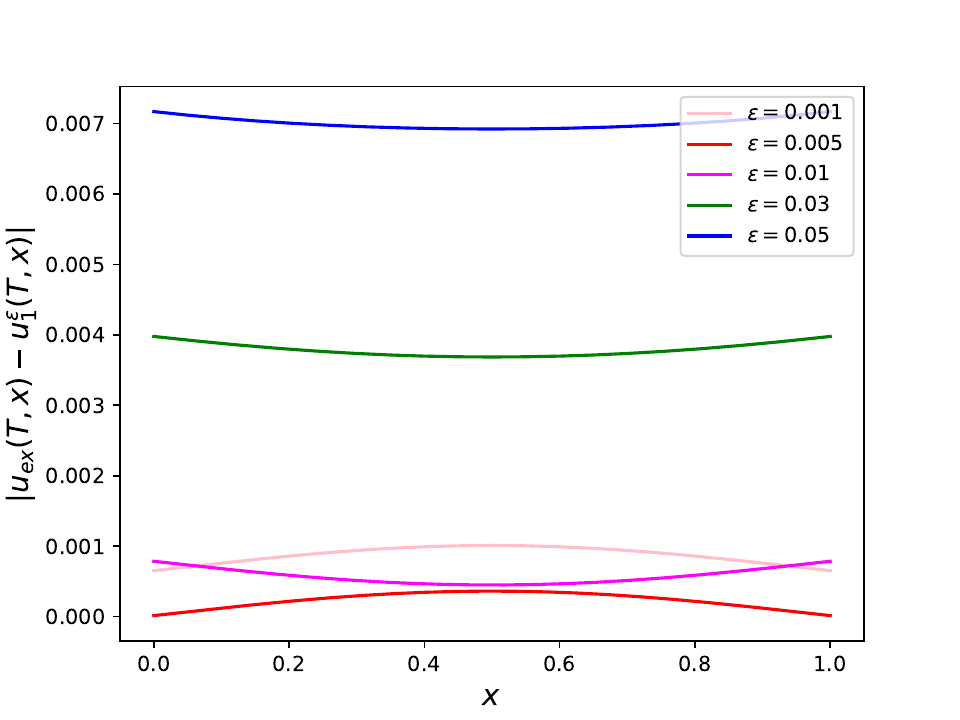}}
\end{center}
\vspace{-0.5cm}
\caption{Experiment 1 (1D): (a)~The exact solution at final time $T=1$ and its numerical approximation, and (b)~its corresponding absolute error, obtained for various levels of noise.}
\label{Experiment1u}
\end{figure}

In the second experiment, we consider the (oscillatory) exact solution
\[
  u_2(t,x) = \cos(2\pi t)\sin(\pi x),
  \qquad
  h_2(t) = \sin(2\pi t),
\]
with
\[
  p_2(x)=-(2\pi+\pi^3)\sin(\pi x), \quad F_2(u)=\pi^2 u,
  \quad
  f_2(t,x)= \pi^2 t\sin(\pi x)\cos(2\pi t),
\]
and
\[
  m_2(t) = \frac{2}{\pi}\cos(2\pi t).
\]
Since the exact measurement $m_2$ is an even function, 
we restrict our approximation to even-degree polynomials.
Based on the improvement analysis in \Cref{tab:degree_selection_1D_2_even}, we select polynomials of degree~6 (i.e., $p^{\epsilon}_{6}(t)$) for Experiment~2. This degree provides the most significant error reduction (up to 95\% improvement from degree~4) while higher degrees provide small improvements (except for $\epsilon= 0.001$). Note that the selected polynomials of degree~6 achieve small $\ell^2$ errors for all noise levels, which we illustrate for $\epsilon = 0.03$ in \Cref{fig:polynomial_analysis_exp2}. The results for Experiment~2 are given in \Cref{Experiment2,Experiment2u}.  It can be seen from the figures that for both one-dimensional experiments accurate approximations 
of the exact sources and solution at final time have been obtained for the different noise levels $\epsilon = \{0.001,0.005,0.01,0.03,0.05\}$. 

\begin{table}[htbp]
\centering
\caption{Polynomial approximation errors $\mathcal{E}(\tilde{d}) = \|p_{\tilde{d}} - m_2^\epsilon\|_{\ell^2}$ and relative improvements $r_{\text{im}}(\tilde{d}) = \frac{\mathcal{E}(\tilde{d}-2) - \mathcal{E}(\tilde{d})}{\mathcal{E}(\tilde{d}-2)}$ for Experiment 2 (1D case). Only even degrees are considered, given the pattern of $m_2.$}
\label{tab:degree_selection_1D_2_even}
\begin{tabular}{lcccccc}
\hline
& & \multicolumn{5}{c}{Noise level $\epsilon$} \\
\cline{3-7}
$\tilde{d}$ & & 0.001 & 0.005 & 0.010 & 0.030 & 0.050 \\
\hline
2 & $\mathcal{E}_2$ & 1.247e-1 & 1.247e-1 & 1.248e-1 & 1.251e-1 & 1.257e-1 \\
%  & $r_2$ & +72.3\% & +72.3\% & +72.3\% & +72.1\% & +71.9\% \\
\hline
4 & $\mathcal{E}_4$ & 1.182e-2 & 1.177e-2 & 1.182e-2 & 1.323e-2 & 1.610e-2 \\
  & $r_{\text{im}}(4)$ & +90.6\% & +90.5\% & +90.5\% & +89.4\% & +87.2\% \\
\hline
6 & $\mathcal{E}_6$ & 5.661e-4 & 1.224e-3 & 2.385e-3 & 7.254e-3 & 1.216e-2 \\
  & $r_{\text{im}}(6)$ & +95.2\% & +89.6\% & +79.8\% & +45.2\% & +24.5\% \\
\hline
8 & $\mathcal{E}_8$ & 2.344e-4 & 1.171e-3 & 2.341e-3 & 7.025e-3 & 1.171e-2 \\
  & $r_{\text{im}}(8)$ & +58.6\% & +4.4\% & +1.8\% & +3.2\% & +3.7\% \\
\hline
10 & $\mathcal{E}_{10}$ & 2.341e-4 & 1.171e-3 & 2.341e-3 & 7.024e-3 & 1.171e-2 \\
  & $r_{\text{im}}(10)$ & +0.1\% & +0.0\% & +0.0\% & +0.0\% & +0.0\% \\
\hline
\end{tabular}
\end{table}

\begin{figure}[htbp]
\begin{center}
\subfigure[]{\includegraphics[width=0.75\textwidth,angle=0,height = 0.22\textheight]{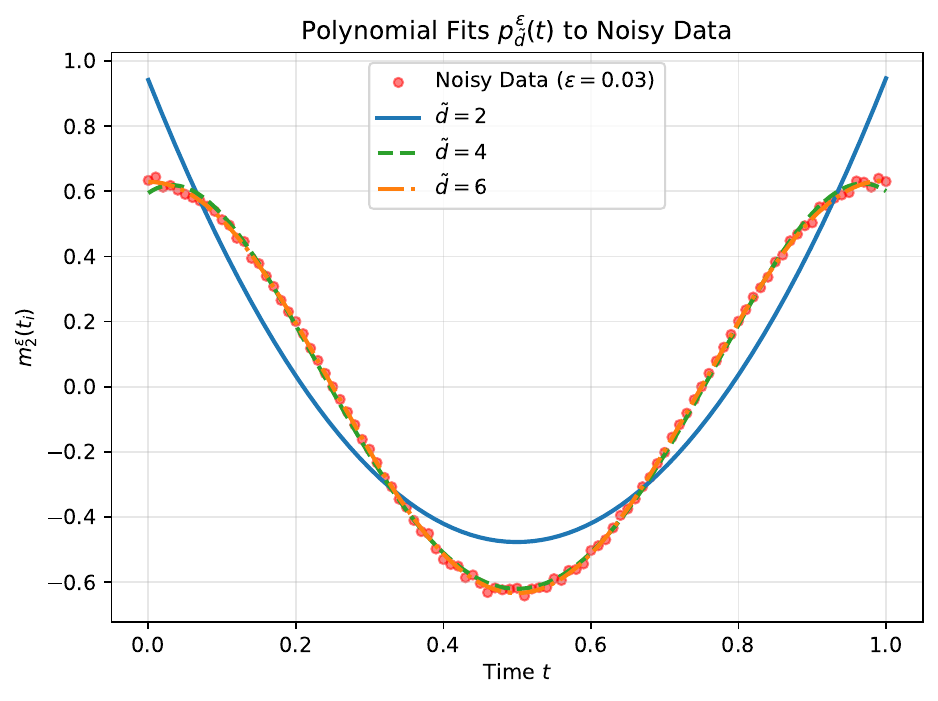}}
\subfigure[]{\includegraphics[width=0.75\textwidth,angle=0,height = 0.22\textheight]{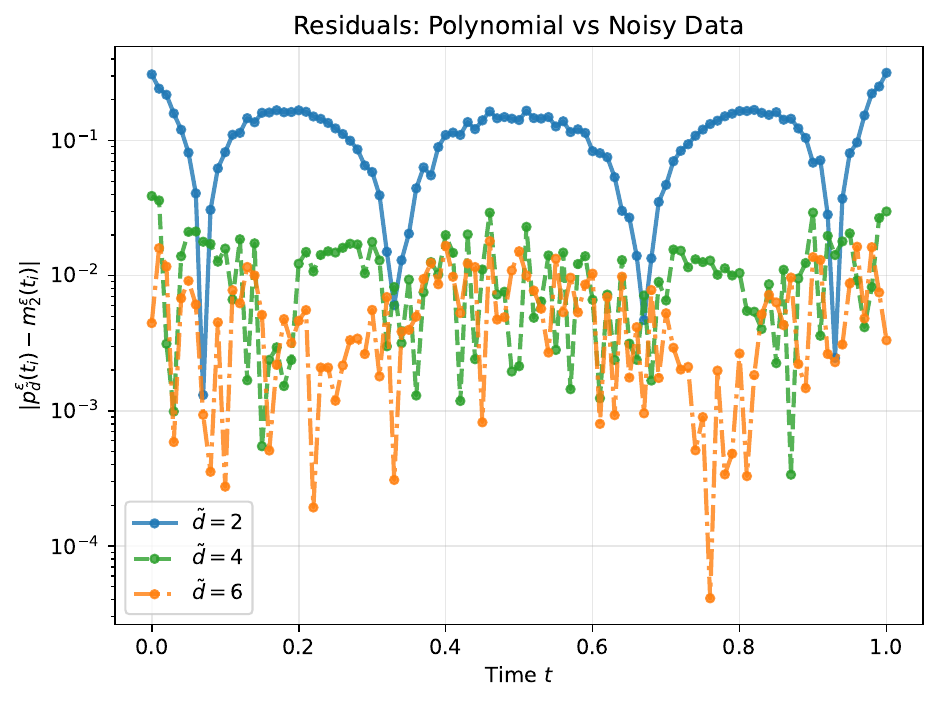}}
\end{center}
\vspace{-0.5cm}
\caption[Polynomial approximation analysis for Experiment 2]{Experiment 2 (1D): (a)~Polynomial approximations of noisy measurement data $m_2^\epsilon(t)$ for $\epsilon = 0.03$, and (b)~corresponding absolute errors. Due to the form of $m_2^\epsilon(t)$, only even degrees are considered. }
\label{fig:polynomial_analysis_exp2}
\end{figure}

%%%%%%%%%%%
\begin{figure}[htbp]
\begin{center}
\subfigure[]{\includegraphics[width=0.75\textwidth,angle=0,height = 0.22\textheight]{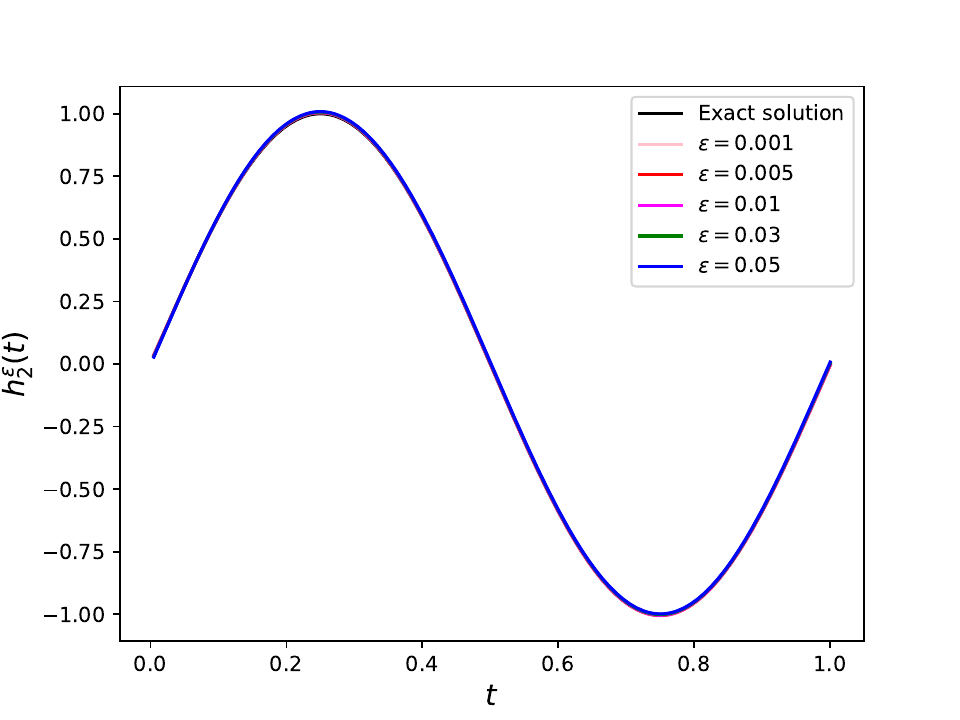}}
\subfigure[]{\includegraphics[width=0.75\textwidth,angle=0,height = 0.22\textheight]{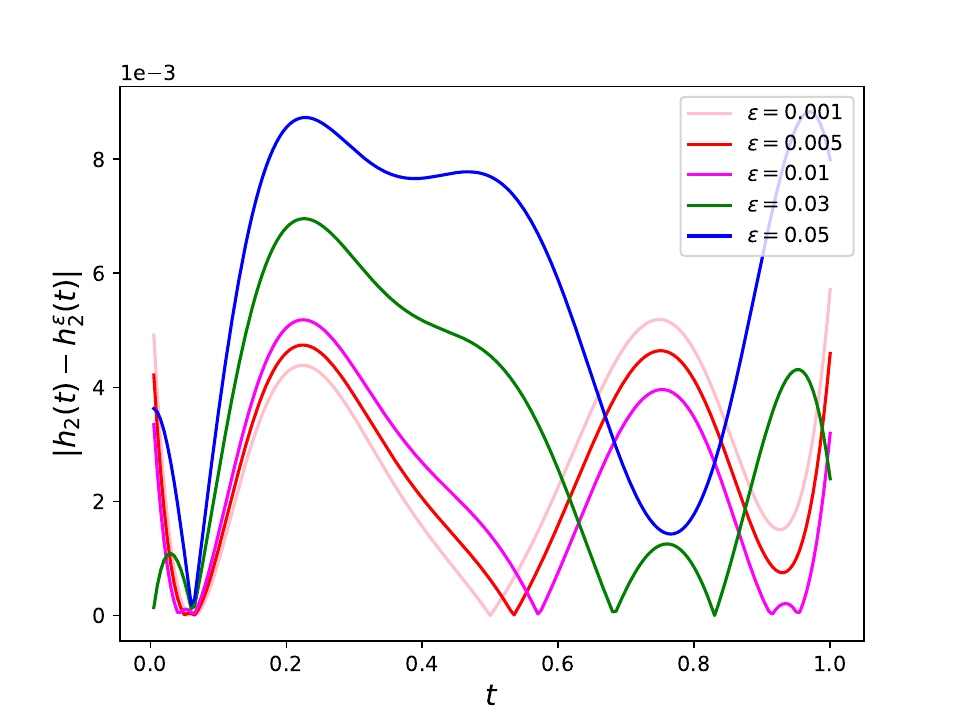}}
\end{center}
\vspace{-0.5cm}
\caption[Unknown time source: Results for Experiment 2]{Experiment 2 (1D): (a)~The exact source and its numerical approximation, and (b)~its corresponding absolute error, obtained for various levels of noise.
}
\label{Experiment2}
\end{figure}

%%%%%%%%%%%
\begin{figure}[htbp]
\begin{center}
\subfigure[]{\includegraphics[width=0.75\textwidth,angle=0,height = 0.22\textheight]{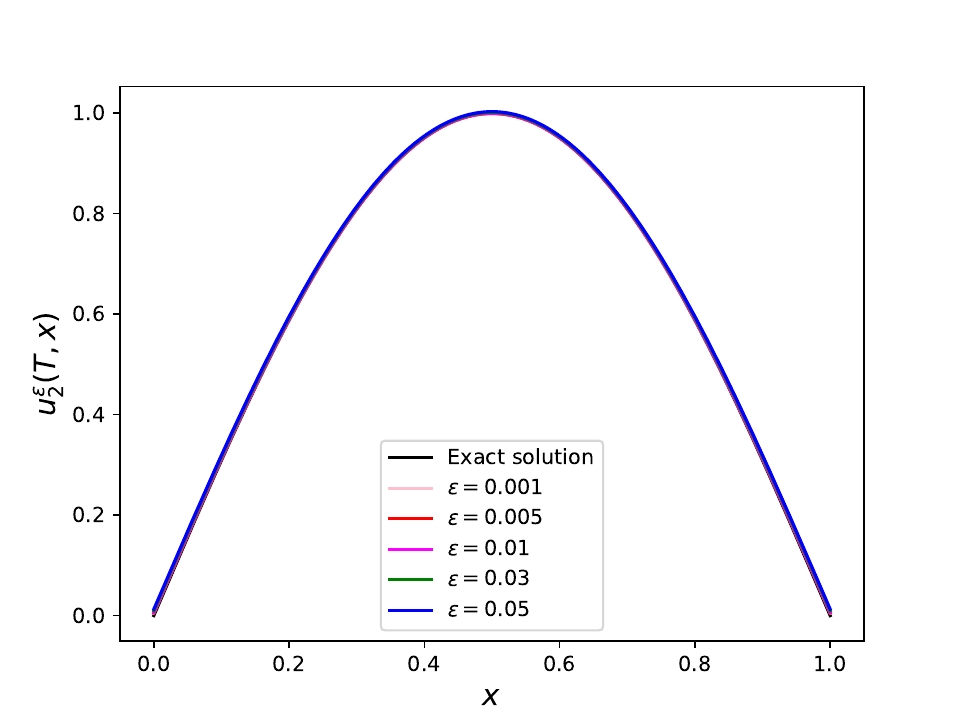}}
\subfigure[]{\includegraphics[width=0.75\textwidth,angle=0,height = 0.22\textheight]{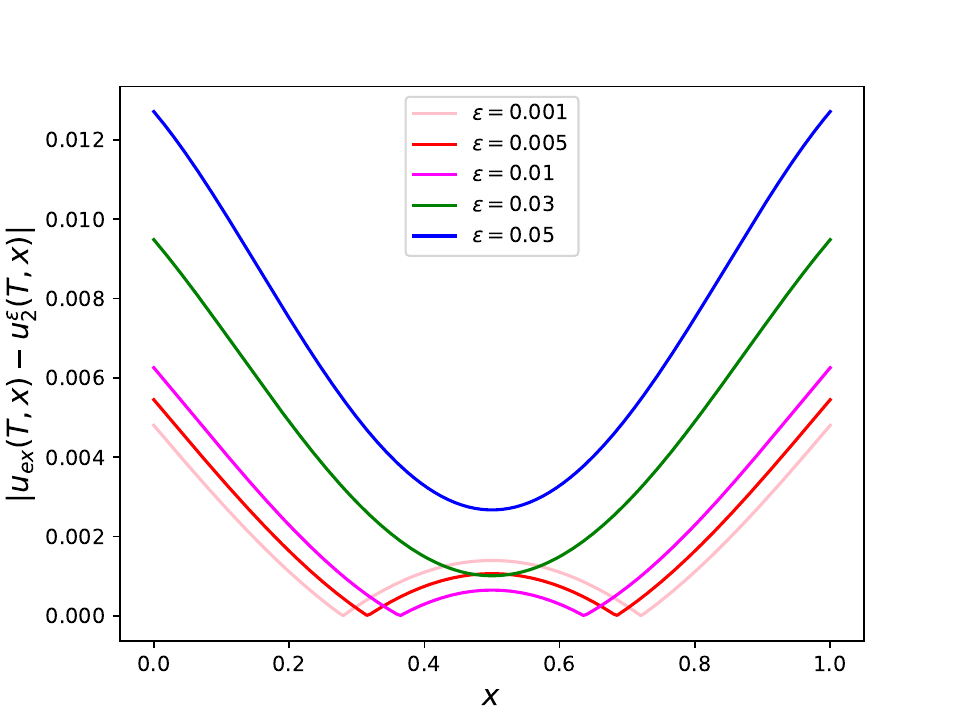}}
\end{center}
\vspace{-0.5cm}
\caption{Experiment 2 (1D): (a)~The exact solution at final time $T=1$ and its numerical approximation, and (b)~its corresponding absolute error, obtained for various levels of noise.}
\label{Experiment2u}
\end{figure}

%%%%%%%%%%%%%%%%%%%%%%%%%%%%%%%%%%%%%%%%%%%%%%%%%%%%%%%%%%%%%%%
\subsection{Two-dimensional experiments with noisy data}
\label{subsec:2dnoisy}
%%%%%%%%%%%%%%%

We remind that $\Omega = (0,1)^2$. 
As a  third experiment, we consider the exact solution 
\[
  u_3(t,x,y)=\exp(-t)\sin(\pi x)\sin(\pi y),\qquad
  h_3(t)=\exp(-t),
\]
with data
\[
   p_3(x,y)=-\pi^2\sin(\pi x)\sin(\pi y), \quad F_3(u)=-u, 
\]
and
\[
 f_3(t,x,y) = \pi^2(t+1)\exp(-t)\sin(\pi x) \sin(\pi y), \quad m_3(t)=\frac{4}{\pi^2}\exp(-t).
\]
The exact solution and data for Experiment~4 are
\[
  u_4(t,x,y)=\cos(2\pi t)\sin(\pi x)\sin(\pi y),\qquad
  h_4(t)=\sin(2\pi t),
\]
with 
\[
  p_4(x)=-2(\pi+\pi^3)\sin(\pi x) \sin(\pi y), \quad F_4(u)=2\pi^2 u,
\]
and
\[
f_4(t,x,y)= 2\pi^2 t\sin(\pi x)\sin(\pi y)\cos(2\pi t), \quad m_4(t) = \frac{4}{\pi^2}\cos(2\pi t). 
\]
It is clear that for both experiments, the polynomial approximation degree will be the same as in the one-dimensional experiments, so degrees three and six, respectively. 
The results of both experiments are shown in  \Cref{Experiment3,Experiment3u,Experiment4,Experiment4u}, 
The obtained results 
demonstrate that our method successfully extends to two spatial dimensions while 
maintaining accurate reconstructions of both the temporal source $h(t)$ and the 
final state $u(1,x,y)$ across all noise levels (only depicted for $\epsilon=0.05)$.

%%%%%%%%%%%
\begin{figure}[htbp]
\begin{center}
\subfigure[]{\includegraphics[width=0.75\textwidth,angle=0,height = 0.22\textheight]{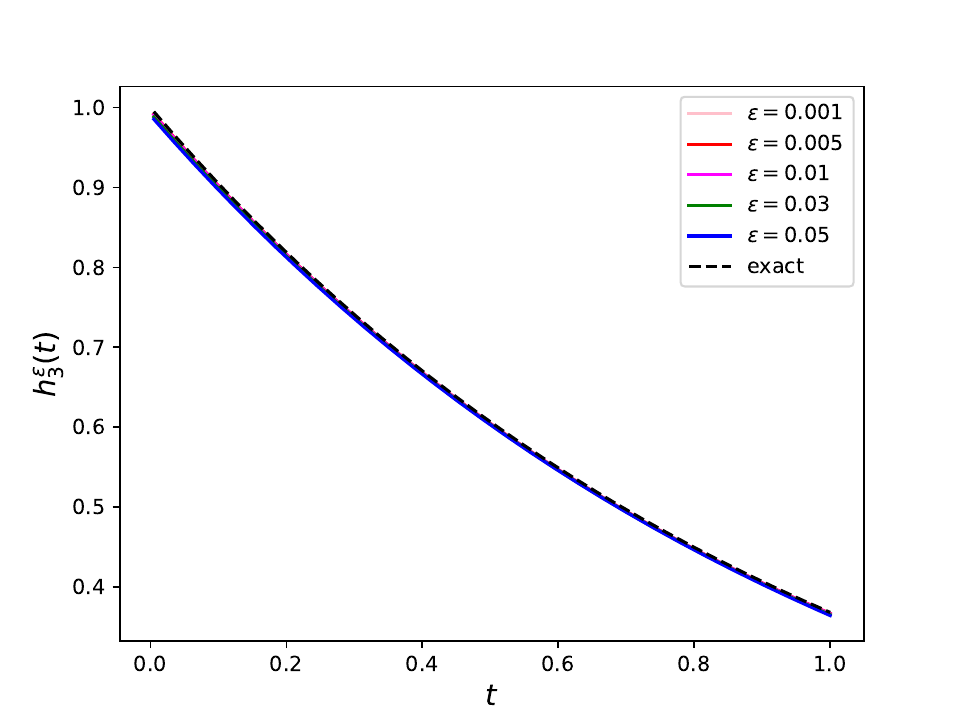}}
\subfigure[]{\includegraphics[width=0.75\textwidth,angle=0,height = 0.22\textheight]{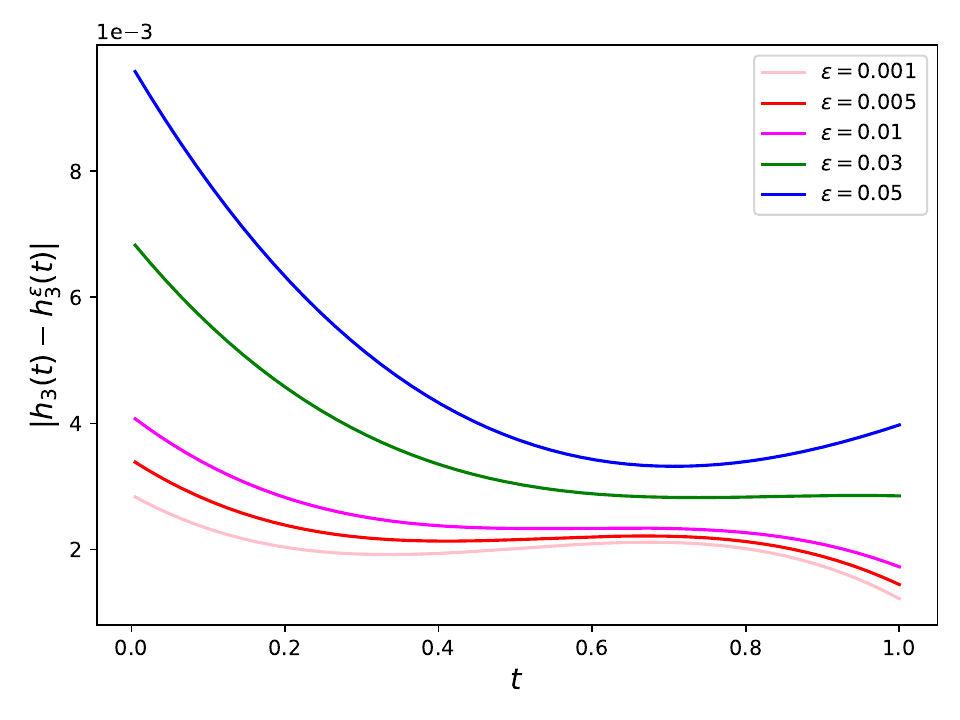}}
\end{center}
\vspace{-0.5cm}
\caption[Unknown time source: Results for Experiment 3]{Experiment 3 (2D): (a)~The exact source and its numerical approximation, and (b)~its corresponding absolute error, obtained for various levels of noise.
}
\label{Experiment3}
\end{figure}

%%%%%%%%%%%
\begin{figure}[htbp]
\begin{center}
\subfigure[]{\includegraphics[width=0.75\textwidth,angle=0,height = 0.22\textheight]{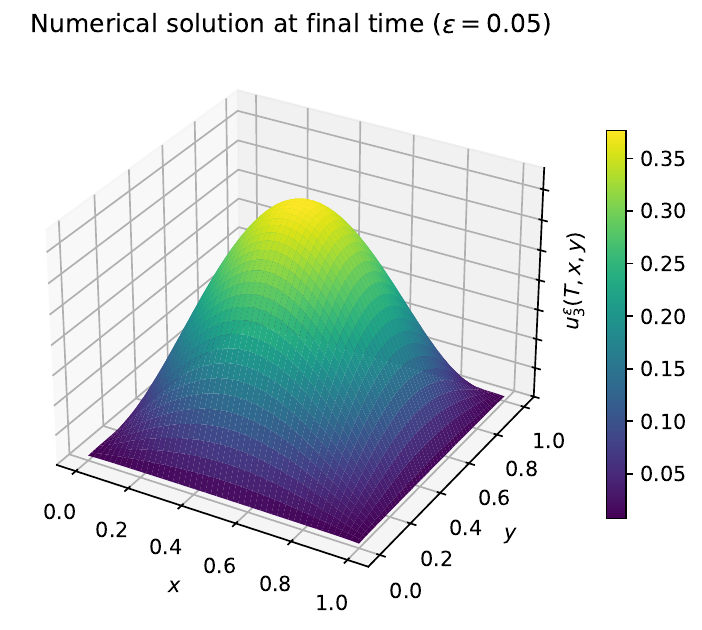}}
\subfigure[]{\includegraphics[width=0.75\textwidth,angle=0,height = 0.22\textheight]{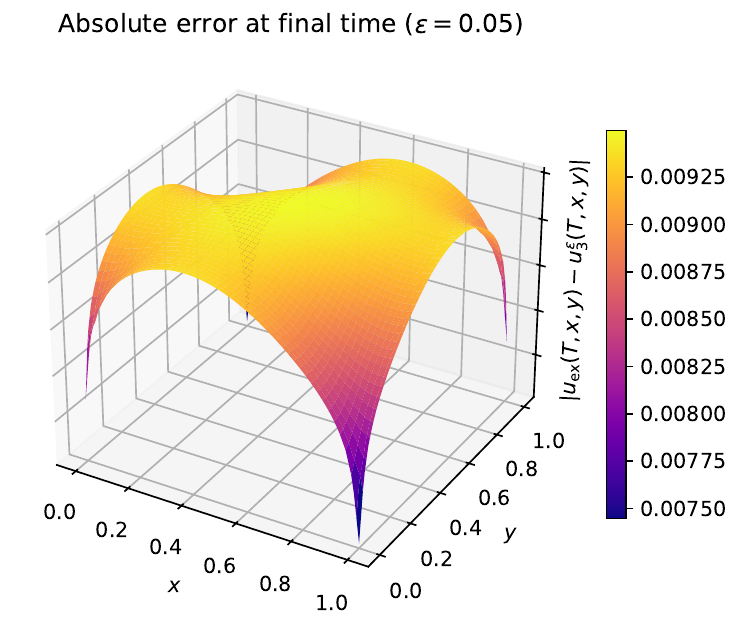}}
\end{center}
\vspace{-0.5cm}
\caption{Experiment 3 (2D): (a)~The exact solution at final time $T=1$ and its numerical approximation for $\epsilon = 0.05$, and (b)~its corresponding absolute error, obtained for various levels of noise.}
\label{Experiment3u}
\end{figure}

%%%%%%%%%%%
\begin{figure}[htbp]
\begin{center}
\subfigure[]{\includegraphics[width=0.75\textwidth,angle=0,height = 0.22\textheight]{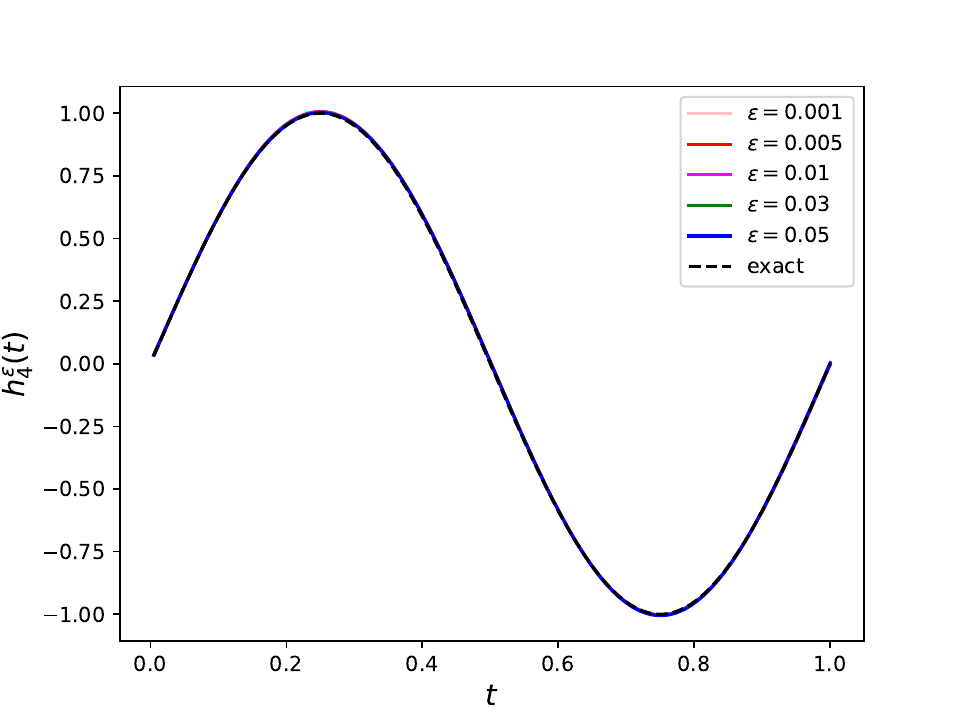}}
\subfigure[]{\includegraphics[width=0.75\textwidth,angle=0,height = 0.22\textheight]{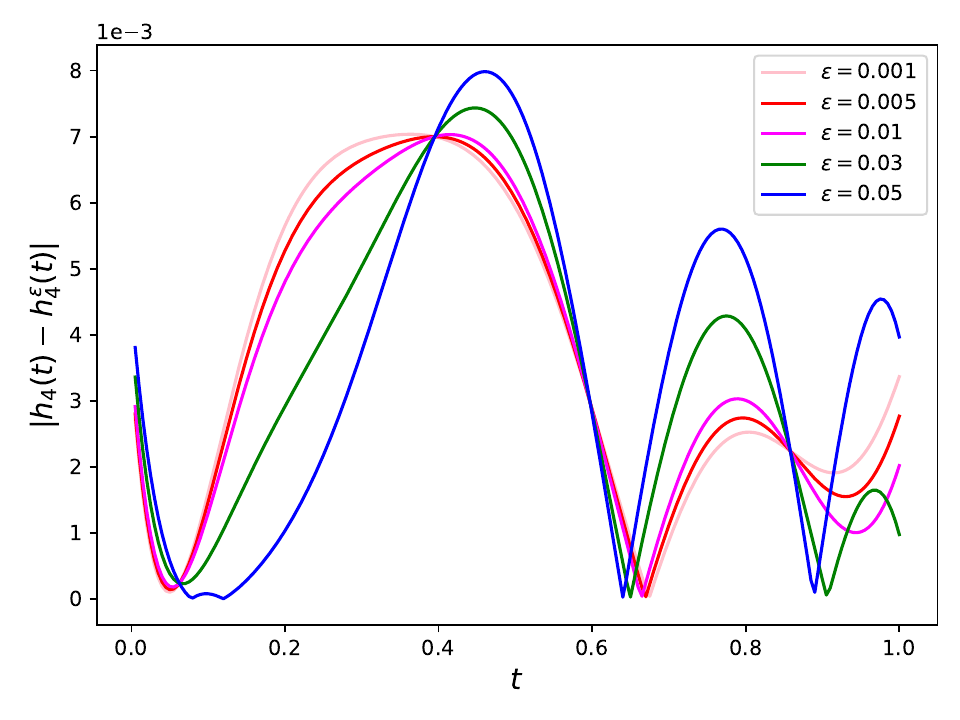}}
\end{center}
\vspace{-0.5cm}
\caption[Unknown time source: Results for Experiment 4]{Experiment 4 (2D): (a)~The exact source and its numerical approximation, and (b)~its corresponding absolute error, obtained for various levels of noise.
}
\label{Experiment4}
\end{figure}

%%%%%%%%%%%
\begin{figure}[htbp]
\begin{center}
\subfigure[]{\includegraphics[width=0.75\textwidth,angle=0,height = 0.22\textheight]{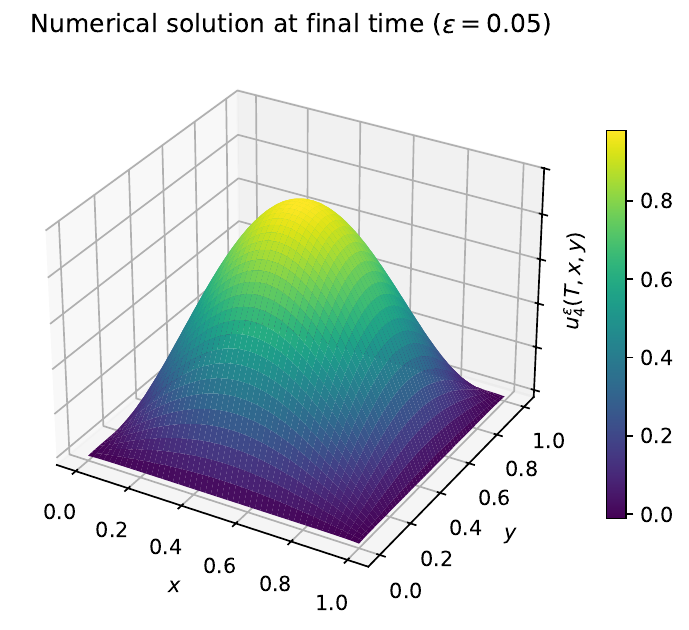}}
\subfigure[]{\includegraphics[width=0.75\textwidth,angle=0,height = 0.22\textheight]{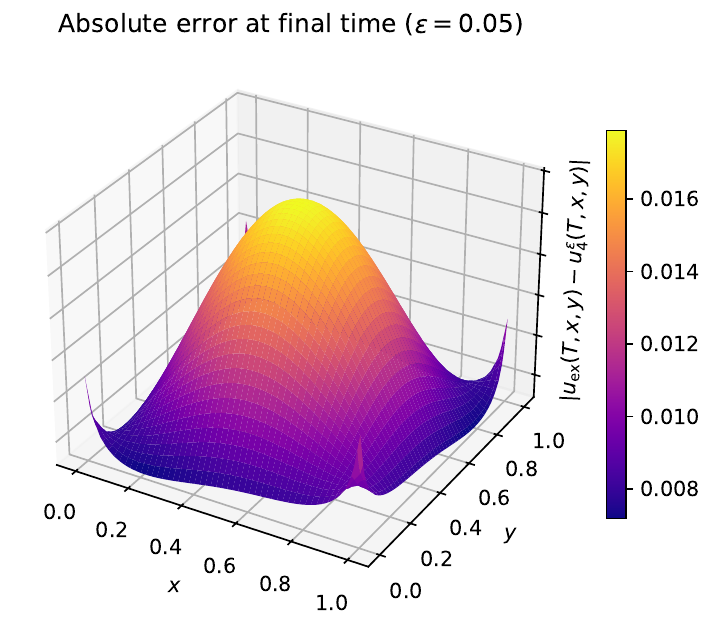}}
\end{center}
\vspace{-0.5cm}
\caption{Experiment 4 (2D): (a)~The exact solution at final time $T=1$ and its numerical approximation for $\epsilon = 0.05$, and (b)~its corresponding absolute error, obtained for various levels of noise.}
\label{Experiment4u}
\end{figure}

%%%%%%%%%%%
\section{Conclusion}
%%%%%%%%%%%

In this study, we addressed the inverse problem of determining an unknown time-dependent source term in a semilinear pseudo-parabolic equation (formulated on a bounded domain with variable diffusion and damping coefficients and a Neumann boundary condition) from an integral overdetermination measurement. By applying Rothe's method, we established the existence and uniqueness of a weak solution under appropriate assumptions on the data. The Neumann boundary condition and the integral measurement over the entire domain were crucial in our approach. Numerically, we developed a robust time-stepping scheme that effectively handles noisy data through polynomial regularisation. More specifically, we regularise $m^\epsilon(t)$ by fitting low-degree polynomials, with the optimal degree selected by analysing relative improvements in the approximation error. Our experiments demonstrate that 
(i) the scheme achieves first-order convergence rates in a noise-free setting;  (ii) successful reconstruction across various noise levels ($\epsilon = 0.001$--$0.05$), also for oscillatory solutions; and (iii) the method extends successfully to two spatial dimensions.
Future work may explore the same inverse problem under a Dirichlet boundary condition and/or the local measurement $\int_\Omega u (t,\X) \omega(\X) \dX = m(t)$.

\section*{Funding} Dr.\ Kh. Khompysh is supported by grant no. AP23486218
 of the Ministry of Science and High Education of the Republic
of Kazakhstan (MES RK).

Dr.\ K. Van Bockstal is supported by the Methusalem programme of the Ghent University Special Research Fund (BOF) grant number 01M01021, and the FWO Senior Research Grant G083525N.

        \bibliography{refs}
		\bibliographystyle{unsrt}%abbrv
\end{document}